\documentclass[reqno,10pt,a4paper]{amsart}
\usepackage{srcltx}
\usepackage{graphicx,color,amssymb,latexsym,amsmath,amsfonts}
\usepackage{mathrsfs}
\usepackage{stmaryrd}
\usepackage{amsmath}
\usepackage{amssymb}
\usepackage{amsthm}
\usepackage{euscript}

\newtheorem{theorem}{Теорема}[section]

\newtheorem{lemma}{Лемма}[section]

\numberwithin{equation}{section}
\def\ov{\overline}

\def\wt{\widetilde}

\def\EE {{\mathbf E}\,}

\def\le{\leqslant}

\def\ge{\geqslant}

\renewcommand{\Im}{\mathop{\mathrm{Im}}\nolimits}
\begin{document}
\title[Принцип больших уклонений]
{Large deviation principle for moderate deviation probabilities of bootstrap empirical measures }
\author{M. S. Ermakov}
\address{Mechanical Engineering Problems Institute RASc\\
Bolshoy pr. V.O., 61\\
199178 St.Petersburg, Russia\\
Saint-Petersburg State University\\
University pr. 28, Petrodvoretz,\\
198504, St.-Petersburg, Russia}
\email{erm2512@mail.ru}
\date{12 november 2012.}

 \thanks{Paper was supported RFFI Grant 11-01-00769-а}


\maketitle

\section{Introduction}
Large deviation principle (LDP) for i.i.d.r.v.'s  ( \cite{ar, bor, ac, dem, ei02, er95,  gor, san}) allows
to study a large number of different problems on large deviation probabilities of statistics. LDP for the bootstrap empirical measures has been studied not in such a large number of papers ( \cite{ch96} and \cite{ch97}).
The goal of the paper is to prove LDP for the conditional distribution of bootstrap empirical measure given empirical measure and similar LDP for the common distribution of bootstrap empirical measure and empirical measure.  To simplify the terminology these LDP will be called moderate deviation principles (MDP) (see \cite{ar}). MDP for the conditional distribution of bootstrap empirical  measure  given empirical measure will be called the conditional principle of moderate deviation probabilities.

For  bootstrap sample means the conditional LDP  has been established in \cite{li}. For bootstrap empirical measures such a version of LDP has been proved in \cite{ch96}.  The strong asymptotics of moderate deviation probabilities of bootstrap sample means have been studied in \cite{gu} and \cite{wo}.

The interest to the problem under consideration is caused the following reasons.

The conditional MDP for bootstrap empirical measures holds for significantly wider zones of moderate deviation probabbilities than MDP for empirical measures. Thus the normal approximation for differentiable statistical functionals depending on bootstrap empirical measures works for significantly wider zones than the normal approximation for the functionals depending on empirical measures.
 At the same time MDP for the common distributions of  bootstrap empirical measures and empirical measures holds for a much narrower zone of moderate deviation probabilities than MDP for empirical probability measures. This result shows significant instability of bootstrap if the empirical probability measure lies in the moderate deviation zone.

As have been shown in
 \cite{er95} and \cite{gao}, MDP for empirical measures and  technique of differentiation in functional spaces allow to establish MDP for differentiable statistical functionals. It turns out that this technique works to the same extent as in the proof of asymptotic normality \cite{we}. The paper allows to obtain similar results for differentiable functionals depending on the bootstrap empirical measures.

Suppose that

--- $S$ is Hausdorff topological space;

--- $\EuScript{F}$ is $\sigma$-field of Borel sets on $S$;

--- $\Lambda$ is the set of all probability measures on
$(S,\EuScript{F})$.

Let $X_1,\ldots,X_n$ be independent identically distributed random variables having probability measure
 ${\mathbf P}\in \Lambda$.

Denote $\widehat{{\mathbf P}}_n$ empirical measure of
$X_1,\ldots,X_n$.

In 1979, in a landmark paper Efron \cite{ef1} proposed to analyze the distributions of statistics $V(X_1,\ldots,X_n)$  with the help of the bootstrap
procedure.
In the bootstrap  we consider the empirical measure $\widehat {\mathbf P}_n$ as
an estimator of the probability measure (pm) ${\mathbf P}$  and simulate the
distribution of statistics $V(X_1,\ldots,X_n)$ on the base of
pm $\widehat {\mathbf P}_n$. In other words, we simulate independent copies
$(X^*_{1i},\ldots,X^*_{ni})_{i\in[1,k]}$ of i.i.d random variables
such that $X^*_{1i}$ is distributed  according to
$\widehat {\mathbf P}_n$.  After that the empirical distribution of
$(V(X^*_{1i},\ldots,X^*_{ni}))_{i\in[1,k]}$ is postulated as an
estimate of the distribution of $V(X_1,\ldots,X_n)$.

It is
of  interest to estimate  large and moderate
deviation probabilities of $V(X_1,\ldots,X_n)$.
Such problems emerge constantly in  confidence estimation and
hypothesis testing. The significant levels in the confidence
estimation and the type I error probabilities in  hypothesis testing are (usually)
of small values and thus are compatible with LDP - MDP
analysis. Hence it appears natural to
compare $V(X_1,\ldots,X_n)$ and $V(X^*_1,\ldots,X_n^*)$ in terms
of LDP - MDP approach.

\medskip
In this paper  we carry out such an MDP based comparisons  in the following setup.

We represent $V(X_1,\ldots,X_n)$ and $V(X^*_1,\ldots,X^*_n)$  as
functionals of $\widehat{{\mathbf P}}_n$ and $ {\mathbf P}^*_n$ respectively,  where $ {\mathbf P}^*_n$ is the
empirical probability measure of $X^*_1,\ldots,X^*_n$ called the bootstrap empirical measure, i.e.
 $$
\begin{aligned}
V(X_1,\ldots,X_n)&= T(\widehat{{\mathbf P}}_n),
\\
V(X^*_1,\ldots,X^*_n)&= T({\mathbf P}^*_n).
\end{aligned}
$$
Thus we reduce the problem to the study of moderate deviation probabilities of
 $T({\mathbf P}^*_n) -T(\widehat{{\mathbf P}}_n)$.

 The paper is organized as follows.

 Theorems \ref{t2.5} and \ref{t2.6} on conditional MDP for the bootstrap empirical measures are provided in section 2.  Theorem \ref{t2.5} states that the conditional MDP holds almost surely. Theorem \ref{t2.6} explores rates of convergence in the conditional MDP.  The results were established in terms of the $\tau_\Theta$ - topology allowing to study moderate deviation probabilities for unbounded statistical functionals.

In section 3 MDP for the common distribution of empirical measures and bootstrap empirical measures are provided. The example given in section 3 shows that the weak $\tau_\Phi$-topology considered in this MDP could not be improved significantly.

 In section  4 we discuss the extensions of these results on the case of differentiable statistical functionals. We show that the technique developed in \cite{er95} and \cite{gao} can be also implemented for the bootstrap setup. In particular MDP for the bootstrap empirical quantile processes and the bootstrap empirical copula functions are provided.

 Probabilities of moderate deviations of statistics have been studied in numerous works (see \cite{al, ar6, in92, jur} and references therein).
Last time this problem was explored in terms of differentiable statistical functionals  \cite{er95, gao}.

In sections 5, 6 and 7 the proofs of Theorems of sections 2 and 3 are provided.

We shall implement the following notations:
\smallskip

-- ${\mathbf Q} << {\mathbf P}$, if ${\mathbf Q} \in \Lambda$ is absolutely continuous with respect to ${\mathbf P} \in \Lambda$;

\smallskip
-- ${\mathbf Q}_2 \times {\mathbf Q}_1$ -- the Cartesian product of probability measures
${\mathbf Q}_2, {\mathbf Q}_1 \in \Lambda$;
\smallskip

-- $\Lambda^2 = \Lambda \times \Lambda$ denote the set of all probability measures
 ${\mathbf Q}_2 \times {\mathbf Q}_1$ with ${\mathbf Q}_2, {\mathbf Q}_1 \in \Lambda$;
\smallskip

-- $C$, $c$ -- positive constants;
\smallskip

-- $\chi(A)$ -- indicator of event $A$;
\smallskip

-- $\int$ denote always $\int\limits_S.$

\section{Conditional MDP for bootstrap empirical measures}

\subsection{The $\tau_\Sigma$-topologies}
Let $\Sigma$ be a set of functions $f: S \to R^1$
such that $\EE[|f(X)|] < \infty$. We suppose that $\Sigma$ contains the set of all bounded functions.

Denote
\begin{equation*}
\Lambda_\Sigma= \left\{{\mathbf P} \in \Lambda: \int |f(X)|d{\mathbf P}<\infty, \quad
f \in \Sigma \right\}.
\end{equation*}
Topology of weak convergence in $\Lambda_\Sigma$ providing the
continuous mapping
$$
{\mathbf Q}\Rightarrow\int f\,d{\mathbf Q}\ \mbox{ for all } \ f\in\Sigma, \ {\mathbf Q}\in\Lambda_\Sigma,
$$
is known as the $\tau_\Sigma$-topology (henceforth,  all topological concepts
 refer to the
 $\tau_\Sigma$-topology). Denote $\sigma_\Sigma$ the smallest $\sigma$-field
that makes all  these mapping measurable.
For any set $\Omega \subset \Lambda_\Sigma$ the notations:
$\mathfrak{cl}(\Omega)$ and $\mathfrak{int}(\Omega)$ are used for the
closure and the interior of $\Omega$ respectively.

For the set $\Sigma=\Theta_0$ of all bounded measurable functions, the
 $\tau_{\Theta_0}$-topology is called the
 $\tau$-topology (see \cite{dem, ei02, gor}).
 Define the set
  $\Lambda_{0\Sigma}$ of all signed measures ${\mathbf G}, {\mathbf G}(S)=0,$ having bounded variation and such that
$$
\int |f| d\, |{\mathbf G}| < \infty.
$$
The measure $|{\mathbf G}|$ is defined as follows. For any set $A\in \EuScript{F}$ , $|{\mathbf G}|(A)$ is variation of set $A$ for signed measure ${\mathbf G}$.

 The
 $\tau_\Sigma$-topology in $\Lambda_{0\Sigma}$  is defined by a standard way.
The definitions of
 $\mathfrak{cl}(\Omega_0)$ and
$\mathfrak{int}(\Omega_0), \Omega_0 \subset \Lambda_{0\Sigma},$ are also standard.

\subsection{Rate function} %
For ${\mathbf G} \in \Lambda_0$, let
$$
\rho^2_0 ({\mathbf G},{\mathbf P})=
 \begin{cases}

\frac{1}{2} \int \left(\frac{d{\mathbf G}}{d{\mathbf P}}\right)^2 d{\mathbf P} , & {\mathbf G}\ll {\mathbf P},
\\
\infty, & \text{otherwise} .
 \end{cases}
$$
be the rate function (in statistical terms, $2\rho_0^2(G|P)$ is
the Fisher information) which arises naturally  in the MDP
analysis of empirical measures $\hat{P}_n$ (see  \cite{bor};  \cite{gao},  \cite{ar} and
  \cite{er95} ).

For the set $\Omega_0 \subset \Lambda_{0\Sigma}$ denote
$$
\rho^2_0 (\Omega_0,{\mathbf P})= \inf\{\rho^2_0 ({\mathbf G},{\mathbf P}), {\mathbf G} \in \Omega_0\}.
$$

\subsection{Outer and inner probabilities}
The empirical distribution function is not measurable (see \cite{ac, gao, leo}).
By this reason, the results will be given in terms of outer and inner probabilities. Let $(\Upsilon, \EuScript{F} ,{\mathbf P})$ be probability space.
The outer probability of set $B \subset \Upsilon$ equals

$$
( {\mathbf P})^*(B) = \inf\{{\mathbf P}(A); B \subseteq A, A \in \Im\},
$$
and its inner probability equals
$( {\mathbf P})_{*}(B)=1 - ( {\mathbf P})^*(\Lambda_{0\Sigma}\setminus B)$.

For a sequence of random variables $Z_n:\Upsilon \to R^1$ ($Z_n$ are not necessary measurable
) we say that $\liminf_{n\to\infty} Z_n \ge c $
inner almost surely ($a.s_*$), if there are measurable random variables  $\Delta_n$,
such that $\Delta_n \le Z_n$ and ${\mathbf P}(\liminf\limits_{n\to\infty} \Delta_n \ge c) =1$.

We say that $\limsup\limits_{n\to\infty} Z_n \le c$ inner almost surely  ($a.s^*.$)  if \break $\liminf\limits_{n\to\infty} -Z_n \ge -c \,\, a.s_*$.

We say that $\limsup\limits_{n\to\infty} Z_n = -\infty$ outer almost surely  ($a.s^*.$) if $\liminf\limits_{n\to\infty} -Z_n \ge -c \,\, a.s_*$ for all $c>0$.

\subsection{Conditional moderate deviation principle for bootstrap empirical measures}
Theorem  \ref{t2.5} given below shows that  MDP holds almost surely (a.s.)
for the conditional distribution of bootstrap empirical
measure given  empirical probability measure.  In this setup we allow  the sample
size $k=k_n$ of the bootstrap   to have  values different from $n$.

The results will be provided in terms of  the $\tau_\Theta$-topologies.

 For each $t> 2$
define the set $\Theta=\Theta_t$ of real functions $f: S \to R^1$ such that
 $E[|f(X)|^t] < \infty$.

 For decreasing function $h :
R^1_+ \to R^1_+ $ and $t \ge 2$ define the set $\Theta=\Theta_{t,h}$ of real functions $f$
such that
\begin{equation}\label{2.13}
P(|f(X)| > s^{-1}) < h(s), \quad s>0
\end{equation}
and
\begin{equation}\label{2.14}
E[|f(X)|^t] < \infty.
\end{equation}

Let
$X^*_1,\ldots,X_{k_n}^*$ be i.i.d.r.v.'s having pm $\hat P_n$.  Denote $P^*_{k_n}$ the
empirical probability measure of  $X_1^*,\ldots,X^*_{k_n}$.
Suppose that $\frac{k_n}{n} < c <\infty$ and $k_n \to \infty$
as $n \to \infty$.
\begin{theorem}\label{t2.5} Let a decreasing sequence $a_n>0, a_n \to 0, a_{n+1}/a_n\to 1,
k_na_n^2 \to \infty$ as $n \to \infty$ be provided. Let
\begin{equation}\label{qqq9}
\sum_{n=1}^\infty h(ca_n) < \infty
\end{equation}
 for all
$c>0$.

Let $\Omega_0 \subset\Lambda_{0\Theta_{2,h}}$. Then
  there hold
 \begin{equation}\label{2.15a}
\liminf_{n\to\infty}(k_na_n^2)^{-1}\ln (\widehat {\mathbf P}_n)_{*} ({\mathbf P}^*_{k_n} \in \widehat {\mathbf P}_n + a_n\Omega_0) \ge
-\rho_0^2 (\mathfrak{int}(\Omega_0),{\mathbf P}) \quad a.\,s_*
\end{equation}
 and
\begin{equation}\label{2.16a}
\limsup_{n\to\infty}(k_na_n^2)^{-1}\ln (\widehat {\mathbf P}_n)^* ({\mathbf P}^*_{k_n}
\!\in\! \widehat {\mathbf P}_n \!+\! a_n\Omega_0)\! \le\!
-\rho_0^2 (\mathfrak{cl}(\Omega_0),{\mathbf P}) \quad a.\,s^*,
\end{equation}
where the closure and the interior of the set $\Omega_0$ in
{\rm(\ref{2.15a})} and {\rm(\ref{2.16a})} are considered with respect to $\tau_{\Theta_{2,h}}$
-topology. The outer probability measure $(\widehat {\mathbf P}_n)^{*}$  and the inner probability measure $(\widehat {\mathbf P}_n)_{*}$ are considered with respect to $\sigma_{\Theta_{2,h}}$-algebra.

Let $\Omega_0 \subset\Lambda_{0\Theta_{t}}$,
$t>2$ and let $a_n = o(n^{-1/t})$. Then
{\rm(\ref{2.15a})} and {\rm(\ref{2.16a})} are valid if $\mathfrak{int}(\Omega_0)$ and $\mathfrak{cl}(\Omega_0)$ are considered with respect to $\tau_{\Theta_{t}}$-topology. Outer probability measure $(\widehat {\mathbf P}_n)^{*}$ and inner probability measure $(\widehat {\mathbf P}_n)_{*}$ are considered with respect to $\sigma_{\Theta_{t}}$-algebra.
\end{theorem}

\subsection{Rates of convergence in conditional moderate deviation principle}
\begin{theorem}\label{t2.6} Let a decreasing sequence $a_n>0, a_n \to 0, a_{n+1}/a_n\to 1,
k_na_n^2 \to \infty$ as $n \to \infty$ be given. Let function $h :
R^1_+ \to R^1_+ $ be such that
\begin{equation}\label{qqq5}
\lim_{n\to\infty} nh(c a_n) = 0
\end{equation}
 for each $c>0$.
Let $\Omega_0 \subset\Lambda_{0\Theta_{t,h}}$, $t>2$. Then
 for any
 $$
 \epsilon > 0\ \mbox{ and } \ n > n_0(\epsilon,\{k_i\}_{i=1}^\infty,\Omega_0)
 $$
 there hold
\begin{equation}\label{2.15}
(k_na_n^2)^{-1}\log (\widehat {\mathbf P}_n)_{*} ({\mathbf P}^*_{k_n} \in \widehat {\mathbf P}_n + a_n\Omega_0) \ge
-\rho_0^2 (\mathfrak{int}(\Omega_0),{\mathbf P})- \epsilon
\end{equation}
 and, if $\rho_0^2 (\mathfrak{cl}_{\Theta_{t,h}}(\Omega_0),{\mathbf P}) <\infty$ additionally, then
\begin{equation}\label{2.16}
(k_na_n^2)^{-1}\log (\widehat {\mathbf P}_n)^* ({\mathbf P}^*_{k_n} \in \widehat {\mathbf P}_n + a_n\Omega_0) \le
-\rho_0^2 (\mathfrak{cl}(\Omega_0),{\mathbf P}) + \epsilon
\end{equation}
 on the sets of events having the inner  probabilities more than
$\kappa_n = \kappa_n(\epsilon,\Omega_0)=1 -
C(\epsilon,\Omega_0)[\beta_{1n}+\beta_{2n}]$ where
$\beta_{1n}=n h(\frac{a_n}{\epsilon C_1(\epsilon,\Omega_0)})$ and
$\beta_{2n}= C_2(\epsilon,\Omega_0) n^{1-t/2}$.

If $\rho_0^2 (\mathfrak{cl}(\Omega_0),{\mathbf P}) =\infty$, then for any $L>0$
\begin{equation}\label{2.16ab}
(k_na_n^2)^{-1}\log (\widehat {\mathbf P}_n)^* ({\mathbf P}^*_{k_n} \in \widehat {\mathbf P}_n + a_n\Omega_0) \le
-L
\end{equation}
on the sets of events having the inner  probabilities more than
$\kappa_{1n} = \kappa_{1n}(L,\Omega_0)=1 -
C(L,\Omega_0)[\beta_{1n}+\beta_{2n}]$ with
$\beta_{1n}=n h(\frac{a_n}{ C_1(L,\Omega_0)})$ and
$\beta_{2n}= C_2(L,\Omega_0) n^{1-t/2}$.
\end{theorem}

\section{Moderate deviation principle for the common distributions  of empirical measures and bootstrap empirical measures}
In section we prove MDP for the common distribution of $({\mathbf P}^*_{k_n} - \widehat {\mathbf P}_n)\times(\widehat {\mathbf P}_n - {\mathbf P})$. We suppose that $k_n/n \to \nu$ as $n \to \infty$.

\subsection{Basic definitions}
Define sequence $b_n$ such that
\begin{equation*}
\left.
\begin{array}{ll}
 b_n&\to 0

 \\
 n b_n^2&\to \infty
 \\
 \frac{b_n}{b_{n+1}}&\to 1
\end{array}
\right\}\quad\text{as}\quad n \to\infty.
\end{equation*}
MDP is provided in terms of the $\tau_\Phi$-topology with the set
$\Phi$ of measurable functions $f$ such that
\begin{equation}\label{2.2}
\lim_{n\to \infty}\frac{1}{nb_n^2}\log ( n{\mathbf P}(|f(X)| > b_n^{-1})) =
- \infty.
\end{equation}
Define the $\tau_\Phi$-topology in $\Lambda^2_\Phi$ and $\Lambda^2_{0\Phi}$ as the product of $\tau_\Phi$-topologies.

For any
 $\ov{{\mathbf G}} = {\mathbf G}_2\times {\mathbf G}_1 \in \Lambda_0^2$ the rate function equals
 $$
\rho_{0b}^2(\ov{{\mathbf G}},\,{\mathbf P}) = \nu\, \rho_0^2({\mathbf G}_2,\,{\mathbf P}) +
\rho_0^2({\mathbf G}_1,\,{\mathbf P}).
$$
For any set
$\ov{\Omega}_0 \subset \Lambda_{0\Phi}^2$ denote
$$
\rho^2_{0b}(\ov{\Omega}_0,\,{\mathbf P}) =
\inf \{\,\rho^2_{0b}(\ov{{\mathbf G}}\,,{\mathbf P})\,: \ov{{\mathbf G}} \in \ov{\Omega}_0\, \}.
$$
We fix signed measures ${\mathbf H}, {\mathbf H}_n \in \Lambda_{0\Phi}$ satisfying the following assumptions.
 \vskip 0.3cm

{\bf A.} There hold
$$
{\mathbf P}_n = {\mathbf P} + b_n {\mathbf H}_n \in \Lambda_\Phi,\quad
{\mathbf P} + b_n {\mathbf H} \in \Lambda_\Phi
$$
 and ${\mathbf H}_n \to {\mathbf H}$ as $n \to \infty$ in the $\tau_\Phi$-toplogy.

{\bf B1.} For any $f \in \Phi$
$$
\limsup_{n\to\infty}\sup_{m}(nb_n^2)^{-1}\log\left(nb_n\int\chi(|f(x)|> b_n^{-1})\, d|{\mathbf H}_m|\right)=-\infty.
$$
Define the signed measure ${\mathbf O}\in \Lambda_{0\Phi}$, such that ${\mathbf O}(A)= 0$ for any set $A \in \EuScript{F}$.
For all ${\mathbf G} \in \Lambda_{0\Phi}$ denote $\widetilde {\mathbf G} = {\mathbf O}\times {\mathbf G}$.

\begin{theorem} \label{t2.1} Assume
{\rm A } and {\rm B1}. Let $\ov{\Omega}_0 \subset
\Lambda^2_{0\Phi}$ be $\sigma_\Phi$-measurable set in $\Lambda^2_{0\Phi}$. Then the following MDP holds
\begin{multline*}
\liminf_{n \to \infty} (nb_n^2)^{-1} \log {\mathbf P}_n(
({\mathbf P}^*_n-\widehat {\mathbf P}_n) \times (\widehat{{\mathbf P}}_n-{\mathbf P}_0)\in b_n \ov{\Omega}_0) \\
\ge
-\rho_{0b}^2(\mathfrak{int}(\ov{\Omega}_0-\wt{\mathbf H}), {\mathbf P})
\end{multline*}
and
\begin{multline*}
\limsup_{n \to \infty} (nb_n^2)^{-1} \log
{\mathbf P}_n( ({\mathbf P}^*_n-\widehat {\mathbf P}_n)\times (\widehat{{\mathbf P}}_n-{\mathbf P})\in b_n \ov{\Omega}_0)
\\
\le-\rho_{0b}^2(\mathfrak{cl}(\ov{\Omega}_0-\wt{\mathbf H}), {\mathbf P}).
\end{multline*}
\end{theorem}
Bolthausen \cite{bol}  has proved the Donsker-Varadhan LDP \cite{do} when the laws of random variables converge weakly and a uniform exponential
integration condition is satisfied. Theorem \ref{t2.1}  can be considered as  a version of these results.
 \medskip

 \noindent\textbf{Remark 2.1.} In hypothesis testing, the  type II error probabilities
are often analyzed for the alternatives ${\mathbf P}_n$
converging to the hypothesis ${\mathbf P}$.
 Theorem~\ref{t2.1} allows to study moderate deviation probabilities
 for this setup. The analysis of importance sampling efficiency is also based on MDP with a sequence of p.m.'s
${\mathbf P}_n$ converging to p.m. ${\mathbf P}$ (see \cite{er95}).
Naturally, if we suppose that ${\mathbf H}_n$, ${\mathbf H}$ are absent,
we get usual form of MDP.

The modern form of  LDP-MDP (see.~\cite{ac, gao, leo}) covers the case of unmeasurable sets $ \ov{\Omega}_0$ and is provided in terms of outer and inner probabilities (see
Theorems~\ref{t2.6} and \ref{t2.5}). Theorem~\ref{t2.1} can be also provided in such a form.

Theorem \ref{t2.4} provided below shows that we can not make significantltly larger the zones of moderate  deviation probabilities in Theorem ~\ref{t2.1}. \vskip 0.2cm

\begin{theorem}\label{t2.4} Let random variable $Y = |f(X)|$ satisfies
{\rm(\ref{2.2})}. Let sequences $r_n$ and $e_n$ be such that
$b_n^{-1} < r_ n$, $b_n^{-1}e_n \to \infty$, $ne_n/r_n \to \infty$
as $n \to \infty$ and
\begin{equation}\label{2.10}
\lim_{n \to \infty}(ne_n^2)^{-1}\log \left(n{\mathbf P}\left(Y > r_{n}\right)\right) = 0,
\end{equation}

\begin{equation}\label{2.11}
\lim_{n\to\infty} (r_ne_n)^{-1}\log\frac{ne_n}{r_n}= 0.
\end{equation}

Let $Y_1,\ldots,Y_n$ be independent copies of $Y$ and let
$Y_1^*,\ldots,Y_n^*$ be bootstrap sample obtained from $Y_1,\ldots,Y_n$.
Then
\begin{equation*}
\lim_{n \to \infty} (ne_n^2)^{-1} \log {\mathbf P}\left(\sum_{i=1}^n Y^*_i > ne_n\right) = 0.
\end{equation*}
\end{theorem}

Proof of Theorem \ref{t2.4} are provided in section 7.

\medskip
\noindent\textbf{Example.} Let ${\mathbf P}(Y> t) = \exp\{-t^\gamma\}$,
$0 <\gamma <1$. Then $b_n = o(n^{-\frac{1}{2+\gamma}})$.
By
straightforward calculations, we get that (\ref{2.10}),
(\ref{2.11}) hold for any sequence $r_n =
n^{\frac{1}{2+\gamma}}f_n$, $e_n =
n^{-\frac{1}{2+\gamma}}f_n^{\frac{\gamma}{2}-\delta}$, with $(\log
n)^{\frac{1}{1 + \frac{\gamma}{2}-\delta}}<<f_n <<
n^{\frac{\gamma}{(2+\gamma)(1 + \delta)}}$ and
$0<\delta<\frac{\gamma}{2}$. Therefore we can not improve significantly the moderate deviation zone in Theorem ~\ref{t2.1} for this asymptotic of ${\mathbf P}(Y>t)$.

\subsection{Moderate deviation principle for empirical measure}
Theorem \ref{t2.3} provided below can be considered as a version of moderate deviation principle established in  \cite{ar} for empirical processes. In such a form this MDP has been proved in  \cite{er95} and is provided here for comparison with the bootstrap results.

Define the set $\Psi$ of measurable functions $f:S\to R^1$ such that
\begin{equation} \label{2.6}
\lim_{n\to \infty} (nd_n^2)^{-1}\log ( n{\mathbf P}(|f(X)| > nd_n)) = - \infty
\end{equation}
where $d_n \to 0$, $nd_n^2 \to \infty$, $d_{n+1}/d_n \to 1$ as $n \to \infty$.

Suppose the following.

{\bf B2.} For each $f \in \Psi$, there holds
\begin{equation*}
\lim_{n\to\infty}(nd_n^2)^{-1}\sup_{m}\log\left(nd_n\int\chi(|f(x)|> nd_n)\, d|H_m|\right)=-\infty.
\end{equation*}
Using the reasoning of Lemma 2.5 in \cite{ei03},
we get that B1 and B2 imply
\begin{equation} \label{e1}
{\rm sup}_{m} \int f^2 d |H_m| < \infty
\end{equation}
and (\ref{2.2}) or (\ref{2.6}) implies
\begin{equation} \label{e2}
\int f^2 d {\mathbf P} < \infty.
\end{equation}
In Lemma 2.5 in \cite{ei03}, (\ref{e2}) has been proved, if $d_n$ is decreasing and $n^{1/2}d_n$ is increasing.
Since $d_n/d_{n-1} \to 1$ as $n \to \infty$ we can choose a subsequence $d_{n_k}$ such that $n_k^{1/2}d_{n_k}$ is increasing and $d_{n_k}/d_{n_{k-1}} \to 1$ as $k \to \infty$. After that we can choose a subsequence $d_{n_{k_i}}$ such that $d_{n_{k_i}}$ is decreasing and $d_{n_{k_i}}/d_{n_{k_{i-1}}}\to 1$ as $i\to \infty$. Implementing to the subsequence $d_{n_{k_i}}$ the same reasoning as in the proof of Lemma 2.5 in \cite{ei03} we get (\ref{e2}) without assuming that the sequences $d_n$ and $n^{1/2}d_n$ are monotone.

\begin{theorem}\label{t2.3}
Assume {\rm A} with $\Phi = \Psi$ and ~{\rm B2}.
Let the set $\Omega_0$ is the $\sigma_\Psi$-measurable subset of
~$\Lambda_{0\Psi}$. Then  MDP holds
\begin{equation*}
\liminf_{n \to \infty} (nd_n^2)^{-1} \log {\mathbf P}_n (\widehat {\mathbf P}_n \in {\mathbf P} + d_n \Omega_0)\ge
-\rho^2_{0}(\mathfrak{int}(\Omega_0 - {\mathbf H}), {\mathbf P}_0)
\end{equation*} and
\begin{equation*}
\limsup_{n \to \infty} (nd_n^2)^{-1} \log {\mathbf P}_n (\widehat {\mathbf P}_n \in {\mathbf P} + d_n \Omega_0)\le
-\rho^2_{0}(\mathfrak{cl}(\Omega_0 - {\mathbf H}), {\mathbf P}_0)
\end{equation*}
\end{theorem}

\noindent{\bf Example}. Let $\EE[\exp\{c|f(X_1)|^\gamma\}]< \infty$, for all $f \in \Theta$ with $\gamma>0$. Then there hold
$$
b_n = o\left(n^{-\frac{1}{1+\gamma}}\right), \quad d_n =o\left(n^{-\frac{1-\gamma}{2-\gamma}}\right)
\quad\text{and}\quad a_n = o\left(|\log n|^{-\gamma}\right).
$$
Therefore conditional MDP holds for significantly wider zone than MDP for empirical  measures.

\section{Moderate deviation probabilities of statistical functionals}
 For statistical functionals the technique of Freshet and Hadamard derivatives (see \cite{er95} and \cite{gao}) works for the  proofs of MDPs to the same extent as in the proofs of asymptotic normality.

\subsection{Differentiable statistical functionals}
For statistical functionals having the Freshet derivatives   MDP has been studied in \cite{er95}. For functionals having the Hadamard derivatives MDP technique has been developed in \cite{gao}. Instead of convergence in the weak topology, in these results the differentiability of statistical functionals in some metric space is supposed. If $S=R^d$, the Kolmogorov-Smirnov metric on the set of distribution functions is continuous in the $\tau$-topology (see \cite{gor}). Thus the functionals continuous in Kolmogorov-Smirnov metric satisfies  MDP. This approach has been implemented
in \cite{er95} for the proof of MDP  for $L$ and $M$ statistics having the Freshet derivatives.
In \cite{gao}, the Hadamard differentiability  in KS-metric allows to derive MDP for Kaplan-Meier estimator, empirical quantile processes and empirical copula functions. The continuiuty of KS-metric in the $\tau$-topology allows to replace  MDP for empirical processes  with MDP for empirical probability measures in the reasoning.   MDP for the bootstrap empirical quantile processes and MDP for the bootstrap empirical copula functions provided in the subsequent subsections follows straightforwardly from the continuiuty of KS-distance in the $\tau$-topology, Theorem \ref{t2.5} and Theorem 3.1 in \cite{gao}.

In \cite{er95}, we prove that Kolmogorov-Smirnov metric having some weight function is continuous in the $\tau_\Psi$-topology. A version of this result for the $\tau_{\Theta_t}$-topology will be provided in the subsection.

Suppose that
$$
\int |x|^{t\kappa} d\,{\mathbf P} < \infty
$$
with $t> 2, \kappa>0$.

Define the set $\Upsilon$ of measurable functions $f: R^d \to R^1$ such that
$$
|f(x)| \le C(1 + |x|^\kappa), \quad x \in R^d.
$$
Define the set $\Lambda_\kappa$ of all probability measures ${\mathbf Q}$ such that
$$
\int |x|^\kappa d\,{\mathbf Q} < \infty.
$$
Let $F(x), x \in R^d,$ be distribution function of probability measure ${\mathbf P}$.

For any ${\mathbf P}$ and ${\mathbf Q}$ define the distance
$$
\rho_\kappa({\mathbf P},{\mathbf Q}) = \sup_{x\in R^d} |F_Q(x) - F_P(x)|(1 + |x|^\kappa)
$$
where $F_Q$ and $F_P$ stand for c.d.f.'s of probability measures ${\mathbf Q}$ and ${\mathbf P}$ respectively.

Define the $\rho_\kappa$-topology  in $\Lambda_\kappa$ generated by the distance $\rho_\kappa$. 

\begin{theorem} \label{ta3} The $\rho_\kappa$-topology is coaser than the $\tau_\Upsilon$-topology.
\end{theorem}

The proof of Theorem \ref{ta3} is akin to the proof of Lemma 4.1 in \cite{er95} and is omitted.

\subsection{Bootstrap empirical quantile processes}

Denote $D[a,b], -\infty< a < b < \infty$ the Banach space of all right continuous with left-hand limits functions $f: [a,b] \to R^1$ equipped with uniform norm. Let \break $F(x)$,
$x\in (a,b)$ -- be distribution function of independent identically distributed random variables  $X_1,\ldots,X_n$. Denote $\widehat F_n$  and $F^*_{k_n}$ respectively  empirical distribution functions of
 $X_1,\ldots,X_n$ and
  $X^*_1,\ldots,X^*_{k_n}$.

For any distribution function $G(x)$, $x \in (a,b),$ and any $p \in (0,1)$, denote $(G)^{-1}(p) = \inf\{x:G(x)\ge p\}$.

\begin{theorem}\label{ta5} Let $a_n>0$ be decreasing sequence such that
 $a_n \to 0$, $a_{n+1}/a_n\to 1$, $k_na_n^2 \to \infty$
as $n \to \infty$. Let fixed values of $p$  and $q$ with
 $0<p<q<1$ be provided. Let $F$ have continuous and positive density on interval  $[(F)^{-1}(p)-\epsilon, (F)^{-1}(q)+\epsilon]$
with $\epsilon > 0$. Then, for any set
$\Omega \subset D((F)^{-1}(p),(F)^{-1}(q))$,
we have
\begin{equation*}
\begin{split}&
\liminf_{n\to\infty}(k_na_n^2)^{-1}\log (\widehat {\mathbf P}_n)_{*} ((F^*_{k_n})^{-1} - (\widehat F_n)^{-1}
\in a_n\Omega) \ge
-I_q( \mathfrak{int}(\Omega))\quad a.\,s_*
\end{split}
\end{equation*}
and
\begin{equation*}
\begin{split}&
\limsup_{n\to\infty}(k_na_n^2)^{-1}\log (\widehat {\mathbf P}_n)^{*} ((F^*_{k_n})^{-1} - (\widehat F_n)^{-1}
\in a_n\Omega) \le
-I_q(\mathfrak{cl}(\Omega)) \quad a.\,s^*,
\end{split}
\end{equation*}
where for any set $\Psi\subset D((F)^{-1}(p)-\epsilon,(F)^{-1}(q)+\epsilon)$
\begin{equation*}
\begin{split}&
 I_q(\Psi)=\inf\left\{\rho_0^2({\mathbf Q},{\mathbf P}): {\mathbf Q}\in\Lambda_{0\Theta_2},\, q=\frac{d{\mathbf Q}}{d{\mathbf P}}, -\frac{q((F)^{-1}(x))}{f((F)^{-1}(x))} =\phi(x),\right.\\& \left.\phi(x)\in\Psi, \,x \in [p,q] \right\}.
\end{split}
\end{equation*}
\end{theorem}

\begin{theorem}\label{ta7} Let $b_n>0$ be decreasing sequence such that
 $b_n \to 0$, $b_{n+1}/b_n\to 1$, $k_nb_n^2 \to \infty$ as $n \to \infty$.
Let $k_n/n \to \nu$ as $n \to \infty$. Let  $F$ satisfy the conditions of Theorem ~{\rm\ref{ta5}}. Then, for any sets
$$
\Omega_1 \subset D ((F)^{-1}(p),(F)^{-1}(q)) \ \text{ and } \ \Omega_2 \subset D((F)^{-1}(p),(F)^{-1}(q)),
$$
we have
\begin{equation*}
\begin{split}&
\liminf_{n\to\infty}(n b_n^2)^{-1}\ln ( {\mathbf P})_{*} ((F^*_{k_n})^{-1} - (\widehat F_n)^{-1}
\in b_n\Omega_2, (\widehat F_{n})^{-1} - ( F)^{-1}
\in b_n\Omega_1)\\& \ge
-\nu I_q( \mathfrak{int}(\Omega_2))+I_q( \mathfrak{int}(\Omega_1))
\end{split}
\end{equation*}
and
\begin{equation*}
\begin{split}&
\limsup_{n\to\infty}(n b_n^2)^{-1}\ln ( {\mathbf P})^{*} ((F^*_{k_n})^{-1} - (\widehat F_n)^{-1}
\in b_n\Omega_2, (\widehat F_{n})^{-1} - ( F)^{-1}
\in b_n\Omega_1)\\& \le
-\nu I_q( \mathfrak{cl}(\Omega_2))+I_q( \mathfrak{cl}(\Omega_1)).
\end{split}
\end{equation*}
\end{theorem}

\subsection{Bootstrap empirical copula processes}
Let $(X_1,Y_1),\break\ldots,(X_n,Y_n)$ be independent identically distributed random vectors having probability measure ${\mathbf P}$ defined on $(a,b)\times(c,d) \supset R^2$.
Let $H$ be distribution function of ${\mathbf P}$. The empirical estimator of copula function $C(u,v) = H((F)^{-1}(u),(G)^{-1}(v))$ is defined as  $\widehat C_n(u,v) = \widehat H_n((\widehat F_n)^{-1}(u),(\widehat G_n)^{-1}(v))$ where $\widehat H_n$ and $\widehat F_n$, $\widehat G_n$ are respectively the joint and  marginal distribution functions of observations. The bootstrap empirical copula function is defined similarly  $C^*_n(u,v) = H^*_n((F^*_n)^{-1}(u),\break(G^*_n)^{-1}(v))$ using the observations
$(X^*_1,Y^*_1),\ldots,(X^*_{k_n},Y^*_{k_n})$. Here $(X^*_1,Y^*_1),\ldots,(X^*_{k_n},Y^*_{k_n})$
are distributed with respect to empirical probability measure $\widehat {\mathbf P}_n$ generated by the observations $(X_1,Y_1),\ldots,(X_n,Y_n)$.

For any set $S$ denote $l_\infty(S)$ linear space of all maps $z: S \to R^1$ having the norm $\|z\| = \sup\limits_{s\in S} |z(s)|$.

\begin{theorem}\label{ta6} Let $a_n>0$  be decreasing sequence such that
$a_n \to 0$, $a_{n+1}/a_n\to 1$, $k_na_n^2 \to \infty$ as $n \to \infty$. Let  $0<p_1<q_1<1$ and $0<p_2<q_2<1$ be fixed.
Suppose that  $F$ and $G$ are continuously differentiable on the intervals $[(F)^{-1}(p_1)-\epsilon,(F)^{-1}(q_1)+\epsilon]$ and $[{(G)^{-1}(p_2)-\epsilon},(G)^{-1}(q_2)+\epsilon]$ respectively and  have strictly positive densities  $f$ and $g$ respectively with $\epsilon >0$.
Suppose that there are continuous derivatives  $\partial H/\partial x$ and $\partial H/\partial y$ on the product of intervals
$$
[(F)^{-1}(p_1)-\epsilon,(F)^{-1}(q_1)+\epsilon]\times[(G)^{-1}(p_2)-\epsilon,(G)^{-1}(q_2)+\epsilon].
$$
Then, for any set $\Omega \subset l_\infty([p_1,q_1]\times[p_2,q_2])$, we have
\begin{equation*}
\begin{split}&
\liminf_{n\to\infty}(k_na_n^2)^{-1}\log (\widehat {\mathbf P}_n)_{*} (C^*_{k_n}- \widehat C_n
\in a_n\Omega) \ge
-I_C( \mathfrak{int}(\Omega))\quad a.\,s_*
\end{split}
\end{equation*}
and
\begin{equation*}
\begin{split}&
\limsup_{n\to\infty}(k_na_n^2)^{-1}\log (\widehat {\mathbf P}_n)^{*} (C^*_{k_n} - \widehat C_n
\in a_n\Omega) \le
-I_C(\mathfrak{cl}(\Omega)) \quad a.\,s^*,
\end{split}
\end{equation*}
where, for any set $\Psi\subset l_\infty([p_1,q_1]\times[p_2,q_2])$, there holds
\begin{align*}
I_C(\Psi) = \inf\bigg\{\rho_0({\mathbf Q}): q = \frac{d{\mathbf Q}}{d{\mathbf P}}, \, {\mathbf Q} \in \Lambda_{0\Theta_2},\,
\Phi'_H(\alpha) = \phi,\, \phi \in \Psi,& \\
\alpha(s,t) =\int\limits_{-\infty}^s\int\limits_{-\infty}^t q(x,y) H(dx,dy)\bigg\},&
\end{align*}
with $\Phi'_H$ defined by
$$
\Phi'_H(\alpha)(u,v) =\alpha((F)^{-1}(u),(G)^{-1}(v))
$$
$$
 - \frac{\partial H}{\partial x}((F)^{-1}(u),(G)^{-1}(v))\frac{\alpha((F)^{-1}(u),\infty)}{f((F)^{-1}(u))}
$$
$$
-
\frac{\partial H}{\partial y}((F)^{-1}(u),(G)^{-1}(v))\frac{\alpha(\infty,(G)^{-1}(v))}{g((G)^{-1}(v))}.
$$
\end{theorem}

\medskip
\noindent{\bf Remark.} Versions of Theorem ~\ref{ta5} and \ref{ta6} can be provided also in terms of convergence on probability  (see Theorem ~\ref{t2.6}).

\section{Proofs of Theorems \ref{t2.5} and \ref{t2.6}
}
We begin with the proof of Theorem \ref{t2.6}. The reasoning are based on the proof LDP for empirical measures proposed in  \cite{ac}. For any $r>0$ define the set $\Gamma_{0r} = \{G: \rho_0^2(G: P) < r, G \in \Lambda_{0\Theta}\}$.

\begin{lemma} \label{l5.1} There hold

{\rm(i)} $\Gamma_{0r} \subset \Lambda_{0\Theta}$,

{\rm(ii)} $\Gamma_{0r}$ is $\tau_\Theta$-compact and sequentially
$\tau_\Theta$-compact set in $\Lambda_{0\Theta}$,

{\rm(iii)} the $\tau$ and $\tau_\Theta$-topologies coincide in $\Gamma_{0r}$.
\end{lemma}

\begin{proof} 
The reasoning are akin to the proof of Lemma ~2.1 in
\cite{ei02}. For any signed measure
 $ {\mathbf G} \in \Gamma_{0r}$, any measurable set $A
\subseteq S $ and each $\phi \in \Theta$, we have
\begin{equation}
\label{4.2}
\int\limits_A |\phi|\, d |{\mathbf G}| \le
\alpha \bigg(\int\limits_A \phi^2\, d{\mathbf P} \bigg) +
\alpha^{-1}\bigg(\int\limits_A \left(\frac{d{\mathbf G}}{d{\mathbf P}}\right)^2 d{\mathbf P} \bigg)^2 d{\mathbf P}
\end{equation}

for all $\alpha > 0$. By the definition of $\Gamma_{0r}$, this implies (i), if $A =S$.

Fix $\epsilon > 0$. Let $\alpha = r/\epsilon$ and let $n = n(\epsilon)$ be such that
$$
\frac{r}{\epsilon} \bigg( \ \int\limits_{|\phi|>n} \phi^2\, d{\mathbf P} \bigg) < \epsilon.
$$
Then
$$
\alpha^{-1}\int\limits_{|\phi|>n} \left(\frac{d{\mathbf G}}{d{\mathbf P}} \right)^2 d{\mathbf P}
\le \epsilon.
$$
Therefore, by (\ref{4.2}), we get
$$
\int |\phi|\,d |{\mathbf G}| -
\int\limits_{|\phi|<n} |\phi|\,d |{\mathbf G}| < 2\epsilon
$$
Hence the map $\Gamma_{0r} \ni {\mathbf G} \to \int \phi\,d {\mathbf G} $
is $\tau$-continuous as continuous limit of functions
$$
\int\limits_{|\phi_1|<n} \phi\, d{\mathbf G}.
$$
Therefore the $\tau$ and $\tau_\Theta$-topologies coincide in $\Gamma_{0r}$.

Since the sets ~$\Gamma_{0r}$ are $\tau$-compact and
sequentially $\tau$-compact (see ~\cite{ar, bor}),
these sets are $\tau_\Theta$-compact and sequentially $\tau_\Theta$-compact.
This completes the proof of Lemma~\ref{l5.1}.
\end{proof}

We begin with the proof of upper bound (\ref{2.16}). Denote
$\eta = \rho_0^2(\mathfrak{cl}(\Omega_0),{\mathbf P})$ and fix $\delta,
0<2\delta<\eta$. It is clear that $\Gamma_{0,\eta-\delta} \subset \Lambda_{0\Theta}\setminus\Omega_0$.

For any $f_1,\ldots,f_l \in \Theta$, ${\mathbf G} \in
\Lambda_{0\Theta}$ and $\gamma > 0$, denote
$$
U(f_1,\ldots,f_l,{\mathbf G}, \gamma) =\left\{{\mathbf R}: \left|\int f_i d({\mathbf R}-{\mathbf G})\right| < \gamma, {\mathbf R} \in \Lambda_{0\Theta}, 1\le i \le l\right\}.
$$
Define the linear space
$$
\wt\Lambda_{0\Theta} = \bigg\{{\mathbf G} : {\mathbf G} = \sum_{i=1}^k \lambda_i {\mathbf G}_i, {\mathbf G}_i \in \Lambda_{0\Theta}, \lambda_i \in R^1, 1 \le i \le k,
k=1,2,\ldots\bigg\}.
$$
Define the $\tau_\Theta$-topology in $\wt \Lambda_{0\Theta}$.
It is clear that $ \Lambda_{0\Theta}\subset \wt \Lambda_{0\Theta}$.

Since $\Lambda_{0\Theta}$ is Hausdorff topological space, then the space $\Lambda_{0\Theta}$ is regular
(see Theorem~B2 in \cite{dem}). Therefore, for each
${\mathbf G} \in \Gamma_{0,\eta-\delta}$, there is open set $U(f_1,\ldots,f_l,{\mathbf G},\gamma) \subset \Lambda_{0\Theta}\setminus {\rm cl}\,(\Omega_0)$.
 The set $\Gamma_{0,\eta-\delta}$ is compact. Therefore there is finite covering of
$\Gamma_{0,\eta-\delta}$ by the sets
$$
U_1 =
U(f_{11},\ldots,f_{1l_1},{\mathbf G}_1, c_1),\ldots,U_m
=U(f_{m1},\ldots,f_{ml_m},{\mathbf G}_m, c_m),
$$
where $f_{ij} \in \Theta$,
${\mathbf G}_i \in \Lambda_{0\Theta}$ для $1\le j \le l_i$,
$1 \le i \le m$. Denote $U = \cup_{i=1}^m U_i$.

Therefore, for the proof of (\ref{2.16}), it suffices to estimate left-hand side
$$
\widehat {\mathbf P}_n({\mathbf P}^*_n\notin \widehat{\mathbf P}_n + a_nU) \ge (\widehat {\mathbf P}_n)^*({\mathbf P}^*_n\in \widehat{\mathbf P}_n + a_n\Omega_0).
$$
The problem was reduced to finite dimensional.

For any finite set $H= \{h_1,\ldots,h_m; h_i \in \Theta, 1\le i \le m\}$ and any set $\Psi \subset \Lambda_{0\Theta}$ denote
$$
\Psi_H = \{z = (z_1,\ldots,z_m): z_i = {\mathbf E} [h_i(X)], 1 \le i \le m\}.
$$
For all $i,\,j$, $ 1 \le j \le l_i$, $1\le i \le m$, define
signed measures ${\mathbf F}_{ij}$ having the densities
$\frac{d{\mathbf F}_{ij}}{d{\mathbf P}}= f_{ij} - \EE[ f_{ij}(X)]$. Define linear spaces
$$
L = \bigg\{ {\mathbf F}: {\mathbf F} = \sum_{i=1}^k\sum_{j=1}^{l_i} \lambda_{ij} {\mathbf F}_{ij},
\lambda_{ij}\in R^1, 1 \le j \le l_i, 1 \le i \le m\bigg\}
$$
and
$$
\wt l = \Big\{ f : f = \frac{d{\mathbf F}}{d{\mathbf P}}, {\mathbf F} \in L\Big\}.
$$
Let $h_1,\ldots,h_{m_1},h_{m_1+1},\ldots,h_m$ be linear independent functions in $\wt l$ such that ${\mathbf E} [h_i^2(X)] = 2(\eta - \delta), 1 \le i \le m_1$, and ${\mathbf E} [h_i^2(X)] = 0, m_1 < i \le m$. Define the sets $H = \{h_1,\ldots,h_{m_1}\}$ and $H_1 = \{h_{m_1+1},\ldots,h_m\}$. We have
$$
(\mathfrak{cl}(\Omega_0))_H \cap (\Gamma_{\eta-\delta})_H = \emptyset.
$$
Denote $\Omega_1 = \{{\mathbf G} : {\mathbf G} \in \mathfrak{cl}(\Omega_0), G_{H} \ne 0\}$ and $\Omega_2 = \{{\mathbf G} : {\mathbf G} \in \mathfrak{cl}(\Omega_0), G_{H_1} \ne 0\}$.

It is clear that
$\widehat {\mathbf P}_n({\mathbf P}^*_n\notin \widehat{\mathbf P}_n + a_n\Omega_1) = 0$ almost surely.

Therefore
$$\widehat {\mathbf P}_n({\mathbf P}^*_n\in \widehat{\mathbf P}_n + a_n\Omega_0) \le \widehat {\mathbf P}_n({\mathbf P}^*_{nH}\in \widehat{\mathbf P}_{nH} + a_n\Omega_{1H}) \le
$$
$$
\widehat {\mathbf P}_n({\mathbf P}^*_{nH}\notin \widehat{\mathbf P}_{nH} + a_n (\Gamma_{0,\eta-\delta})_{H}).
$$
Define the sets
$$\widehat\Gamma_{0c}= \left\{f: f = \frac{dF}{d{\mathbf P}},
F \in \Gamma_{0c}\cap L\right\},\quad  c>0.
$$

There is a finite number of functions
$q_1,\ldots,q_{l} \in \widehat \Gamma_{0,\eta-2\delta}$,
such that
$$
\EE[q_i(X)]=0,\quad \EE[q_i^2(X)] = 2(\eta - 2\delta),\quad
1 \le i \le l
$$ and
\begin{equation*}
\widehat\Gamma_{0,\eta-2\delta}\cap L \subset \cap_{i=1}^l V(q_i)\cap L \subset \widehat\Gamma_{0,\eta-\delta} \cap L,
\end{equation*}
where
$$
 V_i=V(q_i)=\left\{{\mathbf G}: \left|\int q_i d{\mathbf G}\right| < 2(\eta - 2\delta), {\mathbf G} \in \Lambda_{0\Theta}\right\}.
$$
Denote
$$
V = \bigcap_{i=1}^k V_i.
$$
Since $\widehat\Gamma_{0,\eta-\delta} \subset U\cap L$, then $V\subset U$. Therefore
$$
\Omega_1 \subset W = \Lambda_{0\Theta} \setminus V.
$$
Therefore it suffices to estimate the right-hand side
$$
\log (\widehat {\mathbf P}_n)^*({\mathbf P}^*_{k_n} \in \widehat {\mathbf P}_n + a_n \Omega_1)\le\log \widehat {\mathbf P}_n({\mathbf P}^*_{k_n} \in \widehat {\mathbf P}_n + a_n W).
$$
We have
\begin{equation}\label{qq7}
\begin{split}&
\widehat {\mathbf P}_n({\mathbf P}^*_{k_n} \in \widehat {\mathbf P}_n + a_n W) \le \sum_{i=1}^k \widehat {\mathbf P}_n
({\mathbf P}^*_{k_n} \notin \widehat {\mathbf P}_n + a_n U_i)
\\&
=
\sum_{i=1}^k \widehat {\mathbf P}_n \left(\int q_i \,d({\mathbf P}^*_{k_n} -\widehat {\mathbf P}_n) - 2a_n(\eta-2\delta) >0\right).
\end{split}
\end{equation}
Therefore it suffices to prove that, for each  $f \in \Theta, \EE[f(X)]=0$,
$\EE[ f^2(X)]=\eta-2\delta$ and $n > n_0(\epsilon,f)$, we have
\begin{equation}\label{2005}
\begin{split}&
 (k_na_n^2)^{-1}\log \widehat {\mathbf P}_n\left(\int f d({\mathbf P}^*_{k_n} -\widehat {\mathbf P}_n) > 2a_n(\eta-2\delta) \right)\\& \le
-2\frac{(\eta-2\delta)^2}{{\rm Var}\,[f(X_1)]}(1-\epsilon)=-2(\eta-2\delta)(1-\epsilon)
\end{split}
\end{equation}
with probability $\kappa_n(\epsilon,U(f,q))$.

Denote $s^2\doteq s^2_f\doteq s_n^2 = \frac{1}{n}\sum\limits_{i=1}^n
f^2(X_i) - \ov f^2$, where $\ov f=\frac{1}{n}\sum\limits_{i=1}^n
f(X_i)$. We put
 $\gamma = \frac{\sqrt{2}s\epsilon}{324\sigma}$ where $\sigma^2 ={\rm Var}\,[f(X_1)]=\eta- 2\delta$.

By Theorem 28 in \cite[Ch.~4]{pe}, we get ${\mathbf P}(|s_n^2 -
\sigma^2| > \epsilon) < \beta_{2n}(f)$ where
$\beta_{2n}(f)= C_1(f,\epsilon) n^{1-t/2}$.
 Therefore, for the proof of
(\ref{2005}), we can suppose that
\begin{equation}\label{6.3}
|s_n^2 - \sigma^2| < \epsilon.
\end{equation}
Define the set of events
$$
A_{nf}=\{X_1,\ldots,X_n: \max\limits_{1\le s \le n} |f(X_s)| <
\sigma\gamma a_n^{-1}\}.
$$
We have
\begin{equation*}
{\mathbf P}(A_{nf}) \ge 1 - n {\mathbf P}(|f(X_1)| > \sigma\gamma a_n^{-1}) =1
-nh\Big(\frac{a_n}{\sigma\gamma}\Big)\doteq 1 - \beta_{2n}.
\end{equation*}

Note that, by (\ref{qqq5}),
$nh\big(\frac{a_n}{\sigma\gamma}\big) \to 0$ as $n \to \infty$.
Therefore it suffices to prove  (\ref{2005}) if $A_{nf}$ holds.

The further reasoning are based on a slightly simplified version of Theorem ~3.2 in \cite{sau}. This version is provided below.

Let $Y_{1n},\ldots,Y_{k_n,n}$ be independent identically distributed random variables having probability measure ${\mathbf P}_n$, $\EE [Y_{1n}] = 0$, ${\rm Var}\,[Y_{1n}] = \sigma^2$,
$|Y_{in}| < \sigma\gamma a_n^{-1}$.
Denote
$$
S_n = \frac{1}{\sqrt{k_n}\sigma} \sum_{i=1}^{k_n} Y_{in}.
$$
Suppose that
\begin{equation}\label{qqh1}
a_n^{-2} z^{-2} \log {\mathbf E}[\exp\{za_n\sigma^{-1} Y_{1n}\}] <C \quad
\mbox{for all} \quad |z| < \kappa
\end{equation}
and
\begin{equation}\label{uuh2}
\omega = \frac{\sqrt{2}\kappa }{36\max\{1,C\}} > 1.
\end{equation}
Denote
$\Delta = \omega a_nk_n^{1/2}.$
\begin{theorem}\label{t6} Assume {\rm(\ref{qqh1})} and
{\rm(\ref{uuh2})}.
Then we have
\begin{multline}\label{uuh3}
{\mathbf P}(S_n > k_n^{1/2} a_n) \\
=
(1 - \Phi(k_n^{1/2}a_n))\exp\{L(k_n^{1/2}a_n)\}\bigg(1 +
\theta f_1(k_n^{1/2}a_n)\frac{k_n^{1/2}a_n+ 1}{\Delta}\bigg),
\end{multline}
where
$$
f_1(k_n^{1/2}a_n) = \frac{60(1 + 10\Delta^2\exp\{-(1
-\omega^{-1})\sqrt{\Delta}\})} {1 -
\omega^{-1}}
$$
and
\begin{equation}\label{be2}
-\frac{k_n a_n^2}{3\omega} < L(k_n^{1/2}a_n) < \frac{k_n
a_n^2}{2}\frac{1}{1 + \omega}.
\end{equation}
\end{theorem}

Note that, if $\omega > 16$ and $a_nk_n^{1/2} > 100$, then we have
\begin{equation}\label{be1}
|\theta_1 f_1(k_n^{1/2}a_n)|\frac{k_n^{1/2}a_n + 1}{\Delta} < 6.
\end{equation}

If $|z| < \kappa$ and $|f(X_i)| < \sigma\gamma a_n^{-1}$,
$1 \le i \le n$, then
\begin{equation*}
\begin{split}&
\log \EE_{\widehat {\mathbf P}_n}\{\exp\{za_n(f(X^*_1)- \ov f)/s\}\}\\&
=
\log\bigg[\frac{1}{n}\sum_{l=1}^n \exp\{za_n(f(X_i) - \ov
f)/s\}\bigg] \\&
= \log \Big(1\! +\! \frac{z^2a_n^2}{2}\! +\!
\frac{\theta^3 z^3a_n^3s^{-3}}{6n} \sum_{i=1}^n (f(X_i)\! -\! \ov f)^3
\exp\{\theta z a_n(f(X_i)\! -\! \ov f)/s\}\Big)
\\& \doteq \tau_n,
\end{split}
\end{equation*}
where $ 0 <\theta <1 $.

Since
$$
\exp\{\theta z a_n (f(X_1) - \ov f)/s\} <
\exp\{2\gamma\kappa\theta\sigma s^{-1}\}
\doteq R,
$$
using $\log(1+x) < x$, $x > 0$, we get
\begin{equation*}
\tau_n < \log\Big(1 + \frac{z^2a_n^2}{2}(1 + \gamma\kappa\sigma R
s^{-1})\Big) < \frac{z^2a_n^2}{2}(1 + \gamma\kappa\sigma R
s^{-1}) = z^2 a_n^2 D
\end{equation*}
where $D = \frac{1 + \gamma\kappa R\sigma s^{-1}}{2}$.

If
\begin{equation*}
\kappa = \frac{s}{2\gamma\sigma},
\end{equation*}
then $R < 3$ and $D < 2$. Therefore
$$
\omega > \frac{9}{2\epsilon}, \quad L(k_n^{1/2}a_n) \le \frac{k_n^{1/2}a_n^2}{2}\frac{\epsilon}{9/2+\epsilon}.
$$
Hence, by (\ref{uuh3}) and (\ref{be1}), we get
{\allowdisplaybreaks
\begin{align*}
  &
(k_n a_n^2)^{-1}\log \widehat {\mathbf P}_n\left( \int f d({\mathbf P}^*_{k_n} - \widehat {\mathbf P}_n)
> 2a_n (\eta - 2\delta)\right)
\\&
\le
-\frac{1}{2}s^{-2}(\eta - 2\delta)^2\left(1 - \frac{\epsilon}{9/2 +
\epsilon}\right)
\\&
 \quad + (\log 7 - \frac{1}{2}
\log(2\pi s^{-2}(1 + \epsilon))) (k_n a_n^2)^{-1}
\\&
\le
-\frac{1}{2}s^{-2}(\eta - 2\delta)^2\left(1 - \frac{\epsilon}{2}\right) + C(k_na_n^2)^{-1}
\\&
=-\frac{1}{2}s^{-2}(\eta - 2\delta)^2(1-\frac{\epsilon}{2})+ C(k_na_n^2)^{-1}
\\&
\le-\frac{1}{2} s^{-2}(\eta - 2\delta)^2\left(1 -
\frac{\epsilon}{2}\right) + C(k_na_n^2)^{-1}.
 \end{align*}

This implies (\ref{2005}), if (\ref{6.3}) and $|f(X_i)| < \sigma\gamma a_n^{-1}$,
$1\le i \le n$ hold.
This completes the proof of (\ref{2.16}).

If $\rho_0^2(\mathfrak{cl}(\Omega_0),{\mathbf P}) = \infty$, we put $\eta=L$. After that it suffices to implement the same reasoning as in the proof of  (\ref{2.16}).

The proof of lower bound (\ref{2.15}) is based on standard reasoning (see \cite{ac, dem, san} and references therein) and estimates of Theorem \ref{t6}.
 For any $\delta > 0$ there is open set
$U = U(f_1,\ldots,f_l,{\mathbf G},\gamma)$ such that $U \subset \mathfrak{int}(\Omega_0)$ and
$\rho_0^2(U,{\mathbf P}) < \eta + \delta$, $\rho_0^2({\mathbf G},{\mathbf P}) < \eta+\delta$.
Therefore it suffices to find lower bound for the asymptotic
$$
(k_na_n^2)^{-1} \log \widehat {\mathbf P}_n ({\mathbf P}^*_k \in \widehat {\mathbf P}_n + a_n U).
$$
Arguing similarly to the proof of upper bound, we can suppose that the signed measure
 ${\mathbf G}$ has the density $g= \frac{d{\mathbf G}}{d{\mathbf P}}= \sum\limits_{i=1}^l \lambda_i f_i$,
 $f_i \in \Theta$. Thus the problem is finite dimensional.

We fix $\lambda, 0 < \lambda <1$ such that $\lambda {\mathbf G} \in U$.
Note that  $\lambda $ may be defined arbitrary from some vicinity of ~1. Define the set $U_1 = U \cap U(g,{\mathbf G},(1-\lambda)^2\|g\|^2)$. It is clear that we can choose $\lambda$
such that $\rho_0^2(U_1,{\mathbf P}) \le
\frac{1}{2}\lambda^2 \|g\|^2$.

\begin{lemma} \label{l5.2} There is simplex
$\wt U \subset U_1$ bounded the hyperplane
 $$
 \Pi = \big\{{\mathbf R}: \int g \,d{\mathbf R} = \lambda^2 \|g\|^2, {\mathbf R} \in \Lambda_{0\Theta}\big\}
 $$
 and the hyperplanes
$$
\Pi_i = \big\{{\mathbf R}: \int g_i d{\mathbf R} = c_i, {\mathbf R} \in \Lambda_{0\Theta}\big\},
$$ with
 $g_i \in \Theta$, $1 \le i \le l$ such that
$\rho_0^2(\Pi_i,{\mathbf P}) \ge \lambda^2 \|g\|^2> \rho_0^2(\Pi,{\mathbf P})$.
\end{lemma}
The proof of Lemma \ref{l5.2} will be given later.
Let Lemma \ref{l5.2} be valid. Suppose that  $A_{bf}$ holds with $f=g$ and $f = g_i$, $1 \le i \le l$.
 Then, implementing Theorem \ref{t6} and Lemma \ref{l5.2}, we get
{\allowdisplaybreaks
\begin{align}
 &
\widehat {\mathbf P}_n({\mathbf P}^*_{k_n} \in \widehat {\mathbf P}_n + a_nU_1) \ge \widehat {\mathbf P}_n({\mathbf P}^*_{k_n} \in \widehat {\mathbf P}_n + a_n\wt U) \notag
\\ &
\ge
\widehat {\mathbf P}_n\left(\int g d({\mathbf P}^*_{k_n} - \widehat {\mathbf P}_n ) > \lambda^2 \|g\|^2a_n\right)\notag
\\ &
\quad -
\sum_{i=1}^l \widehat {\mathbf P}_n\left(\int g_i (d{\mathbf P}^*_{k_n} - \widehat {\mathbf P}_n) > a_n c_i\right) \label{6.150}
\\ &
\ge
\widehat {\mathbf P}_n\left(\int g d({\mathbf P}^*_{k_n} - \widehat {\mathbf P}_n ) >\lambda^2 \|g\|^2a_n \right)\notag
\\&
 - \sum_{i=1}^l\exp\{-\rho_0^2(\Pi_i,{\mathbf P}) a_n^2 k_n\notag
(1+\epsilon_n)\}
\end{align}}
where $\epsilon_n \to 0$ as $n \to \infty$.

Thus, it remains to implement Theorem  \ref{t6}
to the first addendum of right-hand side of  (\ref{6.150}).

By (\ref{uuh3}) and (\ref{be2}), we get
\begin{equation*}
\begin{split}&
(a_n^2k_n)^{-1} \log \widehat {\mathbf P}_n\left(\int g\, d{\mathbf P}^*_{k_n} - \widehat {\mathbf P}_n) > a_n \lambda^2\|g\|^2\right)
\\&
\ge
-\frac{1}{2} \lambda^2\|g\|^2\left(1 + \frac{1}{3\omega}\right) + c(k_na_n^2)^{-1}
\\&
=
-\frac{1}{2} \lambda^2\|g\|^2\left(1 + \frac{s}{9\sigma}\epsilon\right) + c(k_na_n^2)^{-1}.
\end{split}
\end{equation*}
This implies the lower bound.

\begin{proof}[Proof of Lemma \ref{l5.2}] The problem is reduced to the following. Let we be given a parallelepiped $U_1$ in $R^{l+1}$ and
$0 \notin U_1$. Let the point $u$ lies on the face $\Pi$ of parallelepiped $U_1$ and $|u| = \rho(0,U_1) = \inf_{x \in U_1} |x|$.
 One needs to point out simplex $V \subset U_1$ such that $\Pi \cap V$ is the face of $V$, $ u \in \Pi\cap V$ and, for any hyperplane $\Pi_1$ passing through
another face of V, it holds $\rho(0,\Pi_1) > \rho(0,u)$.
 Let the distance of $u$ from any face other than $\Pi$ exceeds $r_0$.
A simple trigonometric reasoning shows that the simplex $V$ can be defined as follows. We take the vertex $v =(1 + \frac{1}{2} r^2) u$
of $V$ where $ r << r_0$ and all other vertices $v_i, 1 \le i \le l$ belong $\Pi$ and $|v_i - u| =r$.

For the proof of
this statement it suffices to consider the case $l=1$.  Let us draw through $v$ the line $L$
intersecting the line $\Pi$ at the point $w$ and such that
$w$ is orthogonal to $L$. Then $|u-v| = |w - u|^2 |u|^{-1}(1 + o(1))$. Therefore, if the line
$L_1, v \in L_1$ intersect $\Pi$ at the point $z = c |w - u|^2 |u|^{-1},
c <1$, then $\rho(0,L_1) > |u|$.
\end{proof}

\begin{proof}[Proof of Theorem \ref{t2.5}] The reasoning are based on estimates of Theorem \ref{t2.6}.

We begin with the proof of upper bound (\ref{2.16a}) in the case of $\tau_{\Theta_{2h}}$-topology. Suppose that $\rho_0^2(\mathfrak{cl}(\Omega_0),{\mathbf P})< \infty$. If $\rho_0^2(\mathfrak{cl}(\Omega_0),{\mathbf P}) = \infty$, the reasoning are similar. It suffices to prove that, for any  $\epsilon > 0$, there holds
\begin{equation*}
(k_na_n^2)^{-1}\log (\widehat {\mathbf P}_n)^* ({\mathbf P}^*_{k_n} \in \widehat {\mathbf P}_n + a_n\Omega_0) \le
-\rho_0^2 (\mathfrak{cl}(\Omega_0),{\mathbf P})+\epsilon \quad a.\,s^*.
\end{equation*}
By Strong Law of Large Numbers and  (\ref{2.14}), for any $f \in \Theta$
there holds
\begin{equation}\label{v1}
s^2_n(f) \to \sigma^2(f) \quad a.\,s.,
\end{equation}
with $\sigma^2(f) < \infty$.

By (\ref{qqq9}) and (\ref{2.13}), for any $\delta> 0$ we have
\begin{equation}\label{v5}
\begin{split}&
{\mathbf P}(\max_{i\ge l} a_i |f(X_i)| \le \delta ) = \prod_{i=l}^\infty (1- {\mathbf P}(|f(X_i)| > \delta a_s^{-1}))\\&
\ge \prod_{i=l}^\infty (1- h(a_i/\delta)) \ge \exp\left\{-\sum_{i=l}^\infty h(a_i/\delta)\right\} = 1+ o(1)
\end{split}
\end{equation}
as $l \to \infty$.

For each $k$
\begin{equation}\label{v12}
{\mathbf P}(\max_{1\le i\le k} a_n |f(X_i)| > \delta) = o(1) \quad\hbox{as $n \to \infty$}.
\end{equation}

Note that $\max\limits_{i\ge k} a_i |f(X_i)| < \delta
$
implies $\max\limits_{k\le i\le n} |f(X_i)| < \delta a_n^{-1}$. Therefore, by (\ref{v5}) and (\ref{v12}), we get
\begin{equation}\label{v3}
\max_{1\le s\le n} |f(X_s)| < \delta a_n^{-1} \quad a.\,s.
\end{equation}

Using (\ref{v1}) and (\ref{v3}), we can implement the same technique for the proof  of
(\ref{2005}), as in the proof of
(\ref{2.15}) in Theorem ~\ref{t2.6}. This completes the proof of ~(\ref{2.16a}).
\end{proof}

For the proof of (\ref{2.16a}) in the case of  $\tau_{\Theta_{t}}$-topology, it suffices to show that, for any $\delta>0$, there holds
\begin{equation} \label{g1}
I_k\doteq {\mathbf P}(\max_{i>k} a_i |f(X_i)| > \delta) = o(1) \quad\hbox{as $k \to \infty$}.
\end{equation}

We have
\begin{equation*}
I_k \le \sum_{i=k}^\infty {\mathbf P}(f(X_i) > \delta a_i^{-1})
=
\sum_{i=k+1}^\infty (i-k) {\mathbf P}(\delta a_{i-1}^{-1} < |f(X_1)| \le \delta a_i^{-1}) \doteq J_k.
\end{equation*}
Define the function $u(x) =\delta a_{i-1}^{-1}+ \delta(x - a^{-1}_{i -1})$, if
$ x \in [a_{i-1}^{-1}, a_i^{-1})$. Define the inverse function $v(y) = \inf \{t: u(t) = y,
t \in R^1\}$. Define the distribution function  $F(x) = {\mathbf P}(|f(X_1)|<x), x \in R^1_+$.

Then
\begin{equation*}
J_k \le 2 \int\limits_{a_k^{-1}}^\infty v(x) d\,F(x) \le
2\int\limits_{a_k^{-1}}^\infty x^t d\,F(x)= o(1) \quad\hbox{при $k \to \infty$}.
\end{equation*}
This implies (\ref{g1}).

The proof of lower bound (\ref{2.15a})  is based on similar reasoning and is omitted.

\section{Proof of Theorem \ref{t2.1} }

For all $r > 0$ define the set
 $$\Gamma_r = \big\{\ov {\mathbf G} \in \Lambda_0^2:
\rho_{0b}^2(\ov {\mathbf G} : {\mathbf P}) \le r\big\}.$$

\begin{lemma}
\label{l4.1} Let {\rm(\ref{2.6})} hold. Then

{\rm(i)} $\Gamma_r \subset \Lambda^2_{0\Phi}$,

{\rm(ii)} the set  $\Gamma_r$ is $\tau_\Phi$-compact and sequentially
 $\tau_\Phi$-compact set in $\Lambda^2_{0\Phi}$,

{\rm(iii)} the $\tau$ and $\tau_\Phi$-topologies coincide in $\Gamma_r$.
\end{lemma}

The proof of Lemma \ref{l4.1} is akin to the proof of Lemma  \ref{l5.1} and is omitted.

The same reasoning as in the proof of Lemma \ref{l4.1} can be repeated  in the case of
 $\tau_\Psi$-topology. Thus the sets
 $\Gamma_{0r}$ are $\tau_\Psi$-compacts as well.

In Lemmas 4.2--4.5 given below we suppose that the assumptions of Theorem  \ref{t2.1} are satisfied.

For all $u,v\in R^k$ denote $u'v$ the inner product of vectors $u$ and $v$.
For all $f \in \Phi$ and all ${\mathbf G}\in \Lambda_{0\Phi}$ denote
$\langle f,{\mathbf G}\rangle = \int f\, d{\mathbf G}$.

Let $f_1,\ldots,f_{k_1}, g_1,\ldots,g_{k_2} \in \Phi$ and $ {\mathbf G} \in
\Lambda_{0\Phi}$. Let $\EE [f_i(X)]=0$, $\EE [g_j(X)] =0$ for $1\le
i\le k_1$, $1 \le j \le k_2$. Define covariance matrices
$$
R_f =\{ \EE [f_i(X) f_j(X)]\}_{i,j=1}^{k_1}\ \text{ and } \  R_g = \{ \EE [g_i(X)
g_j(X)]\}_{i,j =1}^{k_2}.
$$
 Denote $\vec{f} = \{f_i\}_{i=1}^{k_1}$,
$\vec{g}=\{g_i\}_{i=1}^{k_2}$ and $\ov g_i =
\frac{1}{n}\sum\limits_{l=1}^n g_i(X_l)$, $1\le i \le k_2$.

By   Dawson-Gartner Theorem (see  \cite{dem} Theorem 4.6.9 and \cite{leo}),
Theorem \ref{t2.1} follows from Lemma \ref{l4.2} given below. Note that the de Acosta \cite{ar} approach (see section 5)
also allows to deduce Theorem \ref{t2.1} from Lemma \ref{l4.2}.

\begin{lemma}\label{l4.2} For random vectors
\begin{multline*}
\vec{U}_n( \vec{X}) = \bigg( \frac{1}{n}\sum_{i=1}^n
f_1(X_i),\ldots, \frac{1}{n}\sum_{i=1}^n f_{k_1}(X_i),
\\
\frac{1}{n}\sum_{i=1}^n g_1(X^*_i)- \ov
g_1,\ldots,\frac{1}{n}\sum_{i=1}^n g_{k_2}(X^*_i) - \ov
g_{k_2}\bigg)
\end{multline*}
LDP holds, that is, for every  $\Omega\subset
R^{k_1+k_2}$, there holds
\begin{equation}\label{4.3}
\liminf_{n\to\infty} (nb_n^2)^{-1} \log {\mathbf P}_n(\vec{U}_n(\vec{X}) \in b_n\Omega) \ge -
\inf_{x \in {\rm int} (\Omega)} x'I_{f,g}x
\end{equation}
 and
 \begin{equation}\label{4.4}
\limsup_{n \to \infty} (nb_n^2)^{-1} \log {\mathbf P}_n(\vec{U}_n(\vec{X})
\in b_n\Omega) \le - \inf_{x \in {\rm cl} (\Omega)} x'I_{f,g}x
\end{equation}
where, for each $x =(y,z) \in R^{k_1+k_2}$, $y \in R^{k_1}$ and $z \in R^{k_2}$,
$$
x' I_{f,g} x = \sup_{t\in R^{k_1}, s\in R^{k_2}} \left(t'y +s'z -
\langle t'f,H\rangle - \frac{1}{2} t'R_ft -
\frac{1}{2}s'R_g s\right).
$$
Note that, if there is  $R^{-1}_f$ and $R^{-1}_g$, then
$$
x' I_{fg}x = \frac{1}{2}\big(y -\langle f,H\rangle\big)' R_f^{-1}
\big(y - \langle f,H\rangle \big) + \frac{1}{2} z'R^{-1}_g z.
$$
\end{lemma}

Lemma \ref{l4.2} follows from Lemmas  \ref{l4.3} and \ref{l4.4} given below.

\begin{lemma}
\label{l4.3} There hold
\begin{equation}\label{4.5}
\lim_{n\to \infty} (nb_n^2)^{-1} \log {\mathbf P}_n \Big(\max_{1\le i \le k_1}
\max_{1 \le l \le n} | f_i(X_l)| > b_n^{-1}\Big) = -\infty
\end{equation}
and
\begin{equation}\label{4.6}
\lim_{n\to \infty} (nb_n^2)^{-1} \log {\mathbf P}_n\Big(\max_{1\le i \le k_2}
\max_{1 \le l \le n} | g_i(X^*_l)| > b_n^{-1}\Big) = -\infty.
\end{equation}
\end{lemma}

\begin{proof} 
We have
$$
{\mathbf P}_n \Big(\max_{1\le i \le k_1} \max_{1 \le l \le n} | f_i(X_l)| >b_n^{-1}\Big)
 \le n\sum_{i=1}^{k_1} {\mathbf P}_n(|f_i(X_1)| > b_n^{-1})
$$
$$
\le
n\sum_{i=1}^{k_1} {\mathbf P}(|f_i(X_1)| > b_n^{-1}) + nb_n\sum_{i=1}^{k_1}
\int\chi(|f_i(X_1)| > b_n^{-1})\, d |{\mathbf H}_n|.
$$
By (\ref{2.2}) and B1, this implies (\ref{4.5}).

Since $g_1,\ldots,g_{k_2} \in \Phi$, the same statement holds for these functions as well and we get $$
{\mathbf P}_n(\max_{1\le i \le k_2} \max_{1 \le j \le n} | g_i(X_j)| > b_n^{-1}) =O(\exp\{-Cnb_n^2\})
$$
for any $C >0$. This implies (\ref{4.6}).

For each $h \in \Phi$ denote $ h_n(x) = h(x) \chi(|h(x)| <
b_n^{-1})$. Denote $\vec{f}_n = \{f_{in}\}_{i=1}^{k_1}$ и
$\vec{g}_n=\{g_{in}\}_{i=1}^k$. Define random vector
\begin{multline*}
\wt
U_n( \vec X) = \bigg( \frac{1}{n}\sum_{i=1}^n f_{1n}(X_i),\ldots,
\frac{1}{n}\sum_{i=1}^n f_{k_1n}(X_i),
\\
\frac{1}{n}\sum_{i=1}^n g_{1n}(X^*_i)- \ov g_{1n},\ldots,
\frac{1}{n}\sum_{i=1}^n g_{k_2n}(X^*_i)- \ov
g_{k_2n}\bigg)
\end{multline*}
with $\ov g_{\rm in} = \frac{1}{n}\sum\limits_{l=1}^n
g_{in}(X_l), 1 \le i \le k_2$. Define the sets of events
\begin{multline*}
W_n =
\Big\{X_1,\ldots,X_n: \max_{1\le i \le k_1} \max_{1 \le j \le n} |
f_i(X_j)| < b_n^{-1},\\ \max_{1\le i \le k_2} \max_{1 \le j \le n} |
g_i(X_j)| < b_n^{-1}\Big\}.
\end{multline*}

Denote $\ov W_n$ the complement of
 $W_n$.
By Lemma \ref{l4.2}, we get
\begin{equation} \label{mi}
\begin{split}&
{\mathbf P}_n(\vec{U}_n(\vec{X}) \in b_n\Omega) \le {\mathbf P}_n(\vec{U}_n(\vec{X}) \in b_n\Omega|\ov W_n)
{\mathbf P}(\ov W_n) +{\mathbf P}(W_n)
\\&
 < {\mathbf P}_n(\vec{U}_n(\vec{X}) \in b_n\Omega|\ov W_n)\exp\{o(nb_n^2)\} + \exp\{-Cnb_n^2(1+o(1))\}
\end{split}
\end{equation}
and
\begin{multline*}
{\mathbf P}_n(\vec{U}_n(\vec{X}) \in b_n\Omega)\ge {\mathbf P}_n(\vec{U}_n(\vec{X}) \in b_n\Omega|\ov W_n){\mathbf P}(\ov W_n)
\\
>
{\mathbf P}_n(\vec{U}_n(\vec{X}) \in b_n\Omega|\ov W_n)\exp\{o(nb_n^2)\}
\end{multline*}
where the constant $C$ in (\ref{mi}) can be choosed arbitrary.

Thus, Lemma \ref{l4.2} follows from Lemma \ref{l4.4} given below.
\end{proof}

\begin{lemma}\label{l4.4} For random vectors
$\wt U_n( \vec X) $ LDP holds, that is,
 {\rm(\ref{4.3})} and {\rm(\ref{4.4})} are valid for $\vec{U}_n( \vec{X})= \wt U_n( \vec{X})$.
\end{lemma}

By Gartner-Ellis Theorem  (see \cite[ Lemma \ref{l4.4}]{dem}),
Lemma \ref{l4.4} follows from Lemma  \ref{l4.5} given below.

\begin{lemma}\label{l4.5} Let $f_i \in \Phi$,
$g_j \in \Phi$ for all $1\le i\le k_1$, $1\le j \le k_2$.
Then
\begin{multline}\label{314}
\lim_{n \to \infty} (nb_n^2)^{-1} \log {\mathbf E}_n \left[\exp \left\{ b_n \sum_{l=1}^n t'\vec{f}_n(X_l) +
b_n \sum_{l=1}^n s'(\vec{g}_n(X^*_l) - \ov g_n)\right\}\right]
\\
=
\langle t'\vec{f},H\rangle - \frac{1}{2}t'R_ft - \frac{1}{2} s'R_g s
\end{multline}
with $\ov g_n = (\ov g_{1n}, \ldots, \ov g_{{k_2}n})$.
\end{lemma}

\begin{proof} We begin with the proof of upper bound  in  (\ref{314}). We have
{\allowdisplaybreaks
\begin{align}
 I_n&={\mathbf E}_n \bigg[\exp\bigg\{ b_n \sum_{l=1}^n t'\vec{f}_n(X_l) +
b_n \sum_{l=1}^n s'(\vec{g}_n(X^*_l) - \ov g_n)\bigg\}\bigg]\notag
\\ &
=
{\mathbf E}_n \bigg[\exp\bigg\{ b_n \sum_{l=1}^n t'\vec{f}_n(X_l) \bigg\}\prod_{l=1}^n {\mathbf E}_{\widehat {\mathbf P}_n}[\exp\{
 s'(\vec{g}_n(X^*_l) - \ov g_n)\}]\bigg]\notag
 \\&
 =
{\mathbf E}_n \bigg[\exp\bigg\{ b_n \sum_{l=1}^n t'\vec{f}_n(X_l)\bigg\}
\bigg(\frac{1}{n} \sum_{l=1}^n \exp\{b_ns'(\vec{g}_n(X_l) - \ov g_n)\}\bigg)^n\bigg]\notag
\\&
\le
{\mathbf E}_n \bigg[\exp \bigg\{ b_n \sum_{l=1}^n t'\vec{f}_n(X_l)\bigg\}\bigg(1 +
\frac{b_n^2}{2n} \sum_{l=1}^n(s'(\vec{g}_n(X_l) - \ov g_n))^2\notag
 \\&
\quad+ C(s,k_2)\frac{b_n^3}{6n}
\sum_{l=1}^n|s'(\vec{g}_n(X_l) - \ov g_n)|^3\bigg)^n\bigg]\notag
\\ &
\le
{\mathbf E}_n \bigg[\exp \bigg\{ b_n \sum_{l=1}^n t'\vec{f}_n(X_l) +
\frac{b_n^2}{2} \sum_{l=1}^n(s'(\vec{g}_n(X_l) - \ov g_n))^2 \notag
\\&
\quad+C(s,k_2) b_n^3
\sum_{l=1}^n |s'(\vec{g}_n(X_l) - \ov g_n)|^3\bigg\}\bigg]
\doteq I_{1n}.\label{4.7}
\end{align}}
The first inequality in (\ref{4.7}) follows from Taylor formula   and
\begin{equation*}
|s'(\vec{g_n}(x) - \ov g_n)| \le |s|\,\,|\vec{g_n}(x) - \ov g_n|< |s| 2k_2^{1/2}b_n^{-1}.
\end{equation*}

Denote $\phi_n(X_l)=s'(\vec{g}_n(X_l) - {\mathbf E}_n[\vec{g}_n(X_1)])$ with
$1 \le l \le n$.

By straightforward calculations, we get
\begin{equation}
\label{4.8}
 \sum_{l=1}^n(s'(\vec{g}_n(X_l) - \ov g_n))^2 =
\sum_{l=1}^n\phi_n^2(X_l) - n(s'\ov g_n - {\mathbf E}_n[s'\vec{g}_n(X_1)])^2.
\end{equation}
We have
\begin{multline}\label{4.9}
\sum_{l=1}^n|s'(\vec{g}_n(X_l) - \ov g_n)|^3
\\
\le
\sum_{l=1}^n|\phi_n(X_l)|^3 +
8n|s'(\ov g_n - {\mathbf E}_n[\vec{g}_n(X_1)])|^3\doteq 8V_1 +8nV_2.
\end{multline}
Since
\begin{equation}
\begin{split}&
|s'(\vec g_n(X_1) - {\mathbf E} [g_n(X_1)])|^3 \le |s|^{3/2} |\vec g_n(X_1) - {\mathbf E}_n
[g_n(X_1)]|^{3/2} \\&
= |s|^{3/2} \bigg( \sum_{j=1}^{k_2}(g_{jn}(X_1) -
{\mathbf E}_n[g_{jn}(X_1))^2\bigg)^{3/2} < 8|s|^3 k_2^{3/2} b_n^{-3},
\end{split}
\end{equation}
then
\begin{equation}\label{4.10}
\begin{split}
& b_n^3|V_1| = b_n^3\sum_{l=1}^n|\phi_n(X_l)|^3\chi(|\phi_n(X_l)|
\le \epsilon b_n^{-1}|s| )
\\&
+b_n^3\sum_{l=1}^n |\phi_n(X_l)|^3
\chi(|\phi_n(X_l)| \ge \epsilon b_n^{-1}|s| )
\\ &
\le
\epsilon |s| b_n^2\sum_{l=1}^n\phi^2_n(X_l)+
8|s|^3 k_2^{3/2}\sum_{l=1}^n\chi(|\phi_n(X_l)| \ge \epsilon b_n^{-1}|s|).
\end{split}
\end{equation}
Implementing Jensen's inequality, we get
\begin{equation}\label{4.11}
V_{2}= n^{-3}\biggl|\sum_{l=1}^n \phi_n(X_l)\biggr|^3 \le n^{-1}\sum_{l=1}^n|\phi_n(X_l)|^3
= n^{-1}V_1.
\end{equation}
Using (\ref{4.8})--(\ref{4.11}), we get
\begin{align}
 I_{1n} &\le {\mathbf E}_n\bigg[\exp\bigg\{b_n\sum_{l=1}^n t'\vec{f}_n(X_l) +
\frac{b_n^2}{2}(1-
 2C(s,k_2)\epsilon_n)\sum_{l=1}^n(\phi_n^2(X_l)\notag
 \\& -\frac{b_n^2}{2n}\bigg(\sum_{l=1}^n \phi_n(X_l)\bigg)^2
 +C(s,k_2)|s|^3\sum_{i=1}^n\chi(|\phi_n(X_l)| \ge \epsilon b_n^{-1}|s|))
 \bigg\}\bigg]\notag
 \\&
 \doteq {\mathbf E}_n [W_{n}]
 \label{4.12}
\end{align}
where $\epsilon=\epsilon_n \to 0$ as $n \to \infty$.
For all $r>0$, define the events
$$
A_{n} = A_{nr}\doteq
\{X_1,\ldots,X_n: s'\ov g_n - {\mathbf E}_n [s'g_n(X_1)] < rb_n\}.
$$
 Denote
$\ov A_{n}$ the complement of $ A_{n}$.

We have
\begin{equation*}
\wt I_n = {\mathbf E}_n[W_{n}\chi(A_n)] + {\mathbf E}_n[W_{n}\chi(\ov A_n)]\doteq U_{1n} + U_{2n}.
\end{equation*}
Let  $A_n$ holds. Then
$$
\frac{r^2b_n^4}{2n}\left(\sum_{l=1}^n \phi_n(X_l)\right)^2=\frac{nb_n^2}{2}(s'\ov g_n - {\mathbf E}_n[s'\vec{g}_n(X)])^2 <
\frac{nr^2b_n^4}{2}.
$$
Therefore we have
\begin{equation*}
\begin{split}
& \log[U_{1n}] \le \log {\mathbf E}_n\left[\exp\left\{b_n \sum_{l=1}^nt'\vec{f}_n(X_l) +
 \frac{b_n^2}{2}\sum_{l=1}^n\phi_n^2(X_l)(1+ 2C(s,k_2)\epsilon)\right.\right.
\\ &
\left.\left.+C(s,k_2)|s|^3\sum_{l=1}^n\chi(|\phi_n(X_l
)| \ge \epsilon b_n^{-1})+O(nr^2b_n^4)
\right\}\right]
\\ &
=
 n\log {\mathbf E}_n\left[\exp\left\{b_n t'\vec{f}_n(X_1) +
 \frac{b_n^2}{2}\phi_n^2(X_1)(1+ 2C(s,k_2)\epsilon)\right.\right.
 \\&
\left.\left.+C(s,k_2)|s|^3\chi(|\phi_n(X_1)| \ge \epsilon b_n^{-1})+O(r^2b_n^4)
\right\}\right].
\end{split}
\end{equation*}
Expanding in the Taylor series, we get
\begin{multline*}
\log U_{1n} \le n\log {\mathbf E}_n\left[1+b_n t'\vec{f}_n(X_1) + \frac{b_n^2}{2}(t'\vec{f}_n(X_1))^2\right.
\\
\left. +\frac{b_n^2}{2}\phi_n^2(X_1)(1+ 2C(s,k_2)\epsilon)+ C(s,t,k_1,k_2)\omega_n
+O(r^2b_n^4)\right]
\end{multline*}
with
\begin{align*}
\omega_n &= \omega_{1n} + \omega_{2n} +\omega_{3n} + \omega_{4n} + \omega_{5n},
\\
\omega_{1n}&= \frac{b_n^3}{6} |t'\vec{f}_n(X_1)|^3, \quad \omega_{2n} = 3\frac{b_n^3}{2}|t'\vec{f}_n(X_1)|\phi_n^2(X_1),
\quad
\omega_{3n} = \frac{b_n^4}{8} \phi_n^4(X_1),
\\
\omega_{4n}& =\frac{b_n^4}{12}(t'\vec{f}_n(X_1))^2\phi_n^2(X_1),
\quad
 \omega_{5n} =\chi(|\phi_n(X_1)| \ge \epsilon b_n^{-1}).
\end{align*}
We have
\begin{multline*}
\omega_{1n}\le b_n^3 |t'\vec{f}_n(X_1)|^3\chi(|t'\vec{f}_n(X_1)|< \epsilon b_n^{-1}) +
\chi(\epsilon b_n^{-1}<|t'\vec{f}_n(X_1)|< b_n^{-1})
\\
\doteq \omega_{1n1} + \omega_{1n2},
\end{multline*}
\begin{multline*}
 \omega_{2n} \le b_n^3|t'\vec{f}_n(X_1)|\phi_n^2(X_1)\chi(|t'\vec{f}_n(X_1)|< \epsilon b_n^{-1})
 \\
+C(s,t,k_1,k_2)\chi(\epsilon b_n^{-1}<|t'\vec{f}_n(X_1)|< b_n^{-1})\doteq \omega_{2n1} + \omega_{2n2},
\end{multline*}
\begin{multline*}
\omega_{3n} \le b_n^4 \phi_n^4(X_1)\chi(\phi_n(X_1) < \epsilon b_n^{-1}) + C\chi(\epsilon b_n^{-1}<
\phi_n(X_1) < c b_n^{-1})
\\
 \doteq \omega_{3n1} + \omega_{3n2},
\end{multline*}
\begin{multline*}
\omega_{4n} \le b_n^4
(t'\vec{f}_n(X_1))^2\phi_n^2(X_1)\chi(|t'\vec{f}_n(X_1)|< \epsilon
b_n^{-1})
\\
+
c\chi(\epsilon b_n^{-1}<|t'\vec{f}_n(X_1)|< b_n^{-1})\doteq
\omega_{4n1} + \omega_{4n2}.
\end{multline*}
Using (\ref{2.2}), we get
$$
{\mathbf E}_n[\omega_{1n1}] \le c\epsilon|t|b_n^2 {\mathbf E}_n[(t'\vec{f}_n(X_1))^2], \quad
{\mathbf E}_n[\omega_{2n1}] \le c\epsilon|t|b_n^2 {\mathbf E}_n [\phi_n^2(X_1)],
$$
$$
{\mathbf E}_n[\omega_{3n1}] \le c\epsilon^2|s|^2b_n^2 {\mathbf E}_n [\phi_n^2(X_1)],
\quad \quad
{\mathbf E}_n[\omega_{4n1}] \le c\epsilon^2|t|^2b_n^2 {\mathbf E}_n[\phi_n^2(X_1)]
$$
and
\begin{equation}\label{2.100}
{\mathbf E}_n[\omega_{5n}] \le\epsilon^{-2}b_n^2{\mathbf E}_n[\phi^2_n(X_1)
\chi(|\phi_n(X_i)|\ge\epsilon b_n^{-1})]= o(\epsilon^{-2}b_n^2),
\end{equation}
\begin{equation}\label{2.101}
\begin{split}&
{\mathbf E}_n[\chi(\epsilon b_n^{-1}<|t'\vec{f}_n(X_1)|< b_n^{-1})]\\
&
\le
\epsilon^{-2}b_n^2 {\mathbf E}_n[|t'\vec{f}_n(X_1)|^2\chi(\epsilon
b_n^{-1}<|t'\vec{f}_n(X_1)|)]= o(\epsilon^{-2}b_n^2)
\end{split}
\end{equation}
where the last inequalities in  (\ref{2.100}) and (\ref{2.101})
follows from A and (\ref{e1}), (\ref{e2}).

Hence, we get ${\mathbf E}_n[\omega_n] = o(b_n^2)$.
Therefore we have
\begin{equation*}
\log(U_{1n}) \le -\frac{nb_n^2}{2} \left(2\langle t'\vec{f},H\rangle - t'R_ft - s'R_g s\right)(1+O(1))\doteq v_n.
\end{equation*}
Implementing the Hoelder's inequality, we get
\begin{equation}\label{10.1}
U_{2n} \le \Big({\mathbf E}_n [W_n^{1+\delta}]\Big)^{\frac{1}{1+\delta}}({\mathbf P}(\ov A_n))^{\frac{\delta}{1+\delta}}.
\end{equation}
Using (\ref{4.12}), we get
\begin{multline*}
{\mathbf E}_n [W_n^{1+\delta}] \le {\mathbf E}_n\bigg[\exp\bigg\{(1+\delta)\bigg(b_n\sum_{i=1}^n t'\vec{f}_n(X_i)
\\
 +b_n^2\sum_{i=1}^n
\phi_n^2(X_i)(1 + 2C(s,k_2)\epsilon) 
+2 C(s,k_2)\sum_{i=1}^n \chi(\phi_n(X_i) > \epsilon b_n^{-1})\bigg)\bigg\}\bigg].
\end{multline*}
Hence, repeating the estimates for п $U_{1n}$, we get
\begin{multline}\label{7.101}
{\mathbf E}_n[W_n^{1+\delta}]
\\
 \le \exp \left\{-\frac{(1+\delta)nb_n^2}{2}
(2<t'\vec{f},H> - t'R_ft - s'R_g s)(1+O(1))\right\}.
\end{multline}
Note that (\ref{2.2}) and B1 implies (\ref{2.6}) and (\ref{2.6}) implies
$$
\lim_{n\to \infty} (nr^2b_n^2)^{-1}\log ( n{\mathbf P}(|f(X)| > rnb_n)) = - \infty
$$
for all $r>1$.

Hence, by Theorem 2.4 in \cite{ar}, we get
\begin{equation}\label{7.102}
\log {\mathbf P}_n(\ov A_n) \le -cr^2nb_n^2.
\end{equation}
By (\ref{10.1})--(\ref{7.102}), we get
\begin{equation*}
U_{2n} = o(U_{1n})
\end{equation*}
if $r$ is sufficiently large. This completes the proof of upper bound for $I_n$.
\end{proof}

The proof of lower bound is based on similar estimates. Denote
\begin{multline*}
B_n = \big\{x_1,\ldots,x_n: |f_{ni}(x_s)| < \epsilon b_n^{-1}, \ |g_{nj}(x_s)| < \epsilon b_n^{-1},\\
1 \le s \le n, \ 1 \le i \le k_1, \ 1 \le j \le k_2\big\}.
\end{multline*}
By (\ref{2.2}), (\ref{2.100}) and (\ref{2.101}), we get
\begin{multline*}
{\mathbf P}_n\big(|f_{ni}(X_1)| > \epsilon b_n^{-1}\big)\\
 < \epsilon^{-2}b_n^2 {\mathbf E}_n [f_{ni}^2(X_1)\chi(|f_{ni}(X_1)| > \epsilon b_n^{-1})]=
o(\epsilon^{-2}b_n^2).
\end{multline*} 

Estimating similarly ${\mathbf P}_n(|g_{ni}(X_1)| > \epsilon b_n^{-1})$, we get
\begin{multline*}
{\mathbf P}(B_n) = \prod_{i=1}^{k_1}(1 - {\mathbf P}(|f_{ni}(X_1)| > \epsilon b_n^{-1}))^n
\prod_{i=1}^{k_2}(1 - {\mathbf P}(|g_{ni}(X_1)| > \epsilon b_n^{-1}))^n
\\
=\exp\{-o(nb_n^2)\}.
\end{multline*}
Hence
$$
 I_n\ge {\mathbf E}_n \bigg[\exp\big\{ b_n \sum_{i=1}^n t'\vec{f}_n(X_i)\big\}
\Big(\frac{1}{n} \sum_{i=1}^n \exp\{b_ns'(\vec{g}_n(X_i) - \ov g_n)\}\Big)^n\chi(B_n)\bigg]
$$ 
$$
=\!
{\mathbf E}_n 
\bigg[\exp\big\{ b_n \sum_{i=1}^n t'\vec{f}_n(X_i)\big\}
\Big(\frac{1}{n} \sum_{i=1}^n \exp\{b_ns'(\vec{g}_n(X_i)\! -\! \ov g_n)\}\Big)^n\Big|\,B_n\bigg]
$$
$$
\times \exp\{-o(nb_n^2)\}
\doteq
I_{2n}\exp\{-o(nb_n^2)\}.
$$ 

Expanding in the Taylor series, we get
\begin{equation*}
\begin{split}
& I_{2n} \ge {\mathbf E}_n \Big[\exp \big\{ b_n \sum_{i=1}^n t'\vec{f}_n(X_i)\big\} \Big(1+
\frac{b_n^2}{2n} \sum_{i=1}^n(s'(\vec{g}_n(X_i) - \ov g_n))^2
\\&\quad
 -C(s,k_2)\frac{b_n^3}{n}
\sum_{i=1}^n |s'(\vec{g}_n(X_i) - \ov g_n)|^3\Big)^n\Big|B_n\Big]\ge
\\&\ge
{\mathbf E}_n \Big[\exp \big\{ b_n \sum_{i=1}^n t'\vec{f}_n(X_i)\big\}
\\&\quad \times\Big(1+
\frac{b_n^2}{2n}(1- 2\epsilon) \sum_{i=1}^n(s'(\vec{g}_n(X_i) - \ov g_n))^2 \Big)^n\Big|B_n\Big]
\doteq I_{3n}
\end{split}
\end{equation*}
where the last inequality follows from
$$
\sum_{i=1}^n|s'(\vec{g}_n(X_i) - \ov g_n)|^3 \le 2\epsilon b_n^{-1}
\sum_{i=1}^n (s'(\vec{g}_n(X_i) - \ov g_n))^2.
$$
Since $\log(1+x) \ge 1 + x-x^2$ for $x>0$, then

\begin{equation*}
\begin{split}&
I_{3n} =
{\mathbf E}_n \bigg[\exp \bigg\{ b_n \sum_{i=1}^n t'\vec{f}_n(X_i)\bigg\}\\& \quad\times \exp\bigg\{n\ln\bigg(1+
\frac{b_n^2}{2}(1- 2\epsilon) \sum_{i=1}^n(s'(\vec{g}_n(X_i) - \ov g_n))^2 \bigg)\bigg\}\bigg|B_n\bigg]
\\&
\ge {\mathbf E}_n \bigg[\exp \bigg\{ b_n \sum_{i=1}^n t'\vec{f}_n(X_i)+
\frac{b_n^2}{2}(1- 2\epsilon) \sum_{i=1}^n(s'(\vec{g}_n(X_i) - \ov g_n))^2
\\&
\quad-\frac{b_n^4}{4n}\bigg( \sum_{i=1}^n(s'(\vec{g}_n(X_i) - \ov g_n))^2\bigg)^2
 \bigg\}\bigg|B_n\bigg]
\\&
\ge
{\mathbf E}_n \bigg[\exp \bigg\{ b_n \sum_{i=1}^n
t'\vec{f}_n(X_i)
\\&
\quad+\frac{b_n^2}{2}(1- 2\epsilon-4\epsilon^2)
\sum_{i=1}^n(s'(\vec{g}_n(X_i) - \ov g_n))^2
 \bigg\}\bigg|B_n\bigg]\doteq I_{4n}
\end{split}
\end{equation*}
where the last inequality follows from
$$
\frac{b_n^4}{4n}\bigg( \sum_{i=1}^n(s'(\vec{g}_n(X_i) - \ov g_n))^2\bigg)^2\le
\epsilon^2b_n^2\sum_{i=1}^n(s'(\vec{g}_n(X_i) - \ov g_n))^2.
$$
Estimating similarly to the proof of upper bound, we get
\begin{multline*}
(nb_n^2)^{-1}\ln I_{4n}
\\
 =-\frac{nb_n^2}{2} \left(-2<t'\vec{f},H> - t'R_ft - (1-2\epsilon-2\epsilon^2)s'R_g s\right)(1+O(1)).
\end{multline*}
Since the choice of $\epsilon>0$ is arbitrary, this completes the proofs of lower bound and Lemma \ref{l4.5}.
\hfill{$\square$} 
\section{Proof of Theorem \ref{t2.4}}

It suffices to show that
\begin{equation}\label{5.0}
- \log {\mathbf P}\bigg(\sum_{i=1}^n Y^*_i > ne_n\bigg) = o(ne_n^2).
\end{equation}
Define the events $A_{ni} = U_{ni} \cup V_{ni}, 1 \le i \le n$, where
$$
U_{ni} = \{ Y_i: |Y_i| <b_n^{-1}\} \ \mbox{ and} \
V_{ni} =\{ Y_i: r_n < Y_i \}.
$$
Denote $A_n = \bigcap\limits_{i=1}^n A_{ni}$.
Using (\ref{2.2}), we get
\begin{equation}
\label{5.1}
 {\mathbf P}(A_n) > 1 - {\mathbf P}\Big(\max_{1\le i \le n} |Y_i| > b_n^{-1}\Big)>
1 - n{\mathbf P}(|Y_1| > b_n^{-1}) = 1 +o(1).
\end{equation}

Denote ${\mathbf P}_{cn}$ the conditional probability measure of  $Y_1$ given \break $Y_1 \in A_{n1}$.

Using (\ref{5.1}), we get
\begin{equation*}
\begin{split}
{\mathbf P}\bigg(\sum_{i=1}^n Y^*_i > ne_n\bigg)& \ge {\mathbf P}\bigg(\sum_{i=1}^n Y^*_i > ne_n|A_n\bigg){\mathbf P}(A_n)
\\&
=
{\mathbf P}_{cn}\bigg(\sum_{i=1}^n Y^*_i > ne_n\bigg)(1+o(1)).
\end{split}
\end{equation*}
Thus, it suffices to prove (\ref{5.0})
for the probability measures ${\mathbf P}_{cn}$ instead of ${\mathbf P}$.
Denote $p_n\! =\! {\mathbf P}_{cn} ( Y_1\! >\! r_{n})$. By (\ref{2.2}),
we get~$np_n \!\to\! 0$ as $n \to \infty$. Define the events
\smallskip

$W_n(k_n)\! =\! \{\,Y_1,\ldots,Y_n: n-k_n$ random variables $Y_1,\ldots,Y_n$ belong $(0,b_n^{-1})$ and
$k_n$ random variables $Y_1,\ldots,Y_n$ belong $(r_n,\infty)\,\}$.

\medskip
Suppose that $k=k_n \to \infty$
as $n \to \infty$ and
\begin{equation} \label{bbu}
\lim_{n\to \infty} k_n np_n = 0, \quad \lim_{n\to \infty}
(r_ne_n)^{-1} \log\frac{ne_n}{r_nk_n} = 0.
\end{equation}

Implementing the Stirling formula, we get
\begin{equation}\label{5.2}
\begin{split}
& v_n \doteq {\mathbf P}_{cn}(W_n(k))= \frac{n!}{(n-k)!k!} p_n^{k} (1-p_n)^{n-k}
\\&
=
(2\pi)^{-1/2}\exp\{(n + 1/2)\log n - (n-k+1/2)\log(n-k)
\\& \quad - (k+1/2)\log k
+
k\log p_n + (n-k)\log(1-p_n)\}(1 + o(1))
\\&
=
\exp\Big\{-(n-k+1/2)\log\frac{n-k}{n(1-p_n)} - k \log\frac{k}{np_n}(1 +o(1))\Big\}
\\&
=
\exp\{-n(1-k/n)(-k/n+p_n)(1+o(1))- k\log [k/(np_n)](1+o(1))\}
\\ &
=
\exp\{(k - np_n - k\log(k/(np_n))(1+o(1))\}
\\&
=
\exp\Big\{-k\log\frac{k}{np_n} (1 + o(1))\Big\}.
\end{split}
\end{equation}
It follows from (\ref{2.10}) and (\ref{5.2}) that we can choose $k=k_n$
such that
\begin{equation}\label{5.3}
 |\log v_n| = O(k_n |\log (np_n)|) = o(ne_n^2).
\end{equation}
Define the random variable $l_n$ which equals the number of $Y^*_i, 1  \le i \le n$ such that $Y_i^* \in (r_n,\infty)$. Denote
$u_n = c\frac{ne_n}{r_n} = c \frac{ne_n^2}{r_ne_n}$ с $c >1$
and put $m_n = [u_n]$. Suppose that $\frac{u_n}{k_n} \to \infty$
as $n \to \infty$. Then, estimating similarly to (\ref{5.2}), we get
\begin{equation}\label{bbq}
{\mathbf P}_c(l_n> u_n| W_n(k_n)) =
\exp\Big\{-u_n\log\frac{u_n}{k_n}(1 + o(1))\Big\}.
\end{equation}
Denote $c_1 = c -1$. Denote $Y^{1*}\le \ldots \le Y^{n*}$
the order statistis of $Y^*_1 , \ldots , Y^*_n$.

The event $\{Y^*_1,\ldots, Y_n^*: \sum\limits_{i=1}^n Y_i^* > ne_n\}$
contains the event
$$
U_n = \Big\{Y_1^*,\ldots,Y_n^*: \sum_{j=1}^{n-m_n} Y^{j*} >
-c_1ne_n, |Y^{j*}| < b_n^{-1},
$$
$$
1 \le j \le n - m_n, Y^{t*} > r_n, n - m_n < t \le n
\Big\},
$$
since, if $U_n$ holds, then we have
$$
\sum_{t=n - m_n -1}^n Y^{t*} > r_n m_n = cr_n\frac{ne_n}{r_n} =
cn e_n.
$$
Thus it suffices to prove that
\begin{equation} \label{ddu}
\log {\mathbf P}_c(U_n) = o(ne_n^2).
\end{equation}
We have
\begin{equation*}
\begin{split}&
{\mathbf P}_c(U_n) \!\ge\! {\mathbf P}_c(l_n =m_n) {\mathbf P}_c\Big( \sum_{i=1}^{n-m_n} Y_i^* \!>\! -c_1 n
e_n, |Y_i^*| \!<\! b_n^{-1}, 1 \!\le\! i\! \le\! n - m_n\Big)
 \\&\ge
{\mathbf P}_c(l_n =m_n|W_n(k_n))\\& \quad \times
{\mathbf P}_c(W_n(k_n)){\mathbf P}_c\Big( \sum_{i=1}^{n-m_n} Y_i^* > -c_1 n e_n, |Y_i^*|
< b_n^{-1}, 1 \le i\le n - m_n\Big) .
\end{split}
\end{equation*}
Denote $q_n = {\mathbf P}_c(|Y_1| < b_n^{-1})$.
Define the conditional probability measure ${\mathbf P}_{b_n}$
of random variable $Y_1$ given $|Y_1|
< b_n^{-1}$.

We have
\begin{equation}\label{ttu}
\begin{split}
{\mathbf P}_c(|Y_1^*| < b_n^{-1}) &= \sum_{i=1}^n \frac{n!}{(n-i)!i!}
q_n^i(1-q_n)^{n-i} \frac{i}{n} \\& = q_n \sum_{i=1}^n
\frac{(n-1)!}{(n-i)!(i-1)!} q_n^{i-1}(1-q_n)^{n-i} =q_n.
\end{split}
\end{equation}
We have
\begin{equation}\label{ddd}
\begin{split}&
 {\mathbf P}_c\Big( \sum_{i=1}^{n-m_n} Y_i^* > -c_1 n e_n|\, |Y_i^*|
< b_n^{-1}, 1 \le i \le n - m_n\Big)\\&\quad
= 1- {\mathbf P}_c\Big(
\sum_{i=1}^{n-m_n} Y_i^* < -c_1 n e_n|\, |Y_i^*| < b_n^{-1}, 1 \le i
\le n - m_n\Big).
 \end{split}
\end{equation}
By Chebyshev inequality, using (\ref{ttu}), we get {\allowdisplaybreaks
\begin{align}
 &
{\mathbf P}_c\Big( \sum_{i=1}^{n-m_n} Y_i^* < -c_1 n e_n|\, |Y_i^*| < b_n^{-1},
1 \le i\le n - m_n\Big) \notag
\\&\le
\frac{n-m_n}{c_1^2(n-m_n)
^2e_n^2} {\mathbf E}_c[{\rm Var}_{\widehat
{\mathbf P}_n}(Y^*_1|\,|Y^*_1| < b_n^{-1})]\notag
\\&
=\frac{q_n^2}{c_1^2(n-m_n)e_n^2} \sum_{t=0}^n C_n^t q_n^t (1-
q_n)^{n-t} {\mathbf E}_{b_n}\notag\\
&
\quad\times\bigg[(n-t)^{-1}\sum_{i=1}^{n-t}\bigg(Y_i -
(n-t)^{-1}\sum_{j=1}^{n-t} Y_j\bigg)^2\bigg] \label{dd2}
\\&=
\frac{q_n^2}{c_1^2(n-m_n)e_n^2} \sum_{t=0}^n
C_n^t q_n^t (1- q_n)^{n-t} \frac{t-1}{t} {\rm Var}_{b_n} [Y]\notag\\
& \le \frac{q_n^2}{c_1^2(n-m_n)e_n^2}
{\rm Var}_{b_n}[Y]\notag
\end{align}
}
and
\begin{equation}\label{dd3}
\lim_{n\to\infty} q_n^2 {\rm Var}_{b_n}[Y] = {\rm Var}\,[Y].
\end{equation}
Using (\ref{5.2}) and (\ref{bbq}), we get
\begin{equation}\label{5.6}
\begin{split}
&{\mathbf P}_c(l_n =m_n|W_n(k_n))
{\mathbf P}_c(W_n(k_n))
\\ &
\quad=\exp\left\{-\frac{cne_n^2}{r_ne_n}\log\frac{ne_n}{r_nk_n} -ck_n\log\frac{k_n}{np_n}(1+o(1))\right\}\\
&\quad= \exp\{-o(ne_n^2)\}
\end{split}
\end{equation}
where the last inequality follows from ~(\ref{bbu}),
(\ref{5.3}). Now
(\ref{ddu}) follows from (\ref{ddd}-\ref{5.6}). This completes the
proof of Theorem \ref{t2.4}. \hfill{$\square$}

\bigskip

Ermakov M. S. Large Deviation Principle for moderate deviation probabilities of empirical bootstrap measure.
\smallskip

We prove two Large deviations principles (LDP) in the
zone of moderate deviation probabilities. First we
establish LDP for the conditional distributions of
moderate deviations of empirical bootstrap measures
given empirical probability measures. Second we establish
LDP for the joint distributions of empirical measure and bootstrap
empirical  measures. Using these LDPs,
similar LDPs for statistical differentiable functionals
can be established. The LDPs
for moderate deviations of empirical quantile processes
and empirical bootstrap copula function are provided as illustration
of these results.

\end{document}